\documentclass[a4paper, 12pt]{amsart}
\usepackage{amssymb}

\usepackage{a4wide, enumerate,graphicx,todonotes}
\usepackage[normalem]{ulem} 
\usepackage{mathrsfs}
\usepackage{tabulary}
\usepackage[colorlinks=true,allcolors=blue]{hyperref}
\usepackage[all]{xy}
\usepackage{amscd}

\usepackage{lmodern,amssymb,latexsym,mathtools,slashed,nicefrac,comment}

\hypersetup{
    colorlinks,
    linkcolor={red!50!black},
    citecolor={blue!50!black},
    urlcolor={blue!80!black}
}

\newcounter{abb}

\newcounter{tab}

\makeatletter
\providecommand\@dotsep{5}
\def\listtodoname{List of Todos}
\def\listoftodos{\@starttoc{tdo}\listtodoname}
\makeatother

\newcommand{\abs}[1]{\left\lvert#1\right\rvert}

\newcommand{\dd}{\mathrm{d}}

\newcommand{\R}{\mathbb{R}}
\newcommand{\CC}{\mathbb{C}}

\newcommand{\eps}{\varepsilon}

\newcommand{\myicon}{$\,\,\,\triangleright$}

\newcommand{\supp}{\mathrm{supp}} 

\newcommand{\Cl}{\mathrm{Cl}}

\newcommand{\tensor}{\otimes}
\newcommand{\vol}{\mathrm{vol}}
 
\newcommand{\im}{\mathrm{im}\,} 
 
\newcommand{\Sym}{\mathop\mathrm{Sym}}
 
\newcommand{\Lip}{\textnormal{Lip}}
\newcommand{\Liploc}{\textnormal{Lip}_{\rm{loc}}}

\newcommand{\Wloc}{W_{\rm{loc}}} 
\newcommand{\Lloc}{L_{\rm{loc}}}

\newcommand{\norm}[1]{\left\|#1\right\|}

\newcommand{\defeq}{\mathrel{\vcentcolon=}}

\DeclareMathOperator{\End}{End}
 
\DeclareMathOperator{\dom}{dom}
\DeclareMathOperator{\rank}{rank}

\DeclareMathOperator{\ind}{index}

\renewcommand{\rm}{\mathrm}

\DeclareMathOperator{\scal}{scal}

\DeclareMathOperator{\id}{\mathrm{id}}

\DeclareMathOperator{\ch}{ch}

\DeclareMathOperator{\Dirac}{\mathsf{D}} 

\DeclareMathOperator{\Mr}{M_{\rm reg}} 
\DeclareMathOperator{\ZZ}{\mathbb{Z}}

\DeclareMathOperator{\cl}{\textnormal{Cl}}

\DeclareMathOperator{\Ker}{Ker}
\DeclareMathOperator{\Prin}{P}

\newcommand{\SO}{\mathrm{SO}}
\newcommand{\Spin}{\mathrm{Spin}}

\def\S{\mathord{\mathbb S}}
\def\D{\mathord{\mathbb D}}

\makeatletter
\DeclareFontFamily{OMX}{MnSymbolE}{}
\DeclareSymbolFont{MnLargeSymbols}{OMX}{MnSymbolE}{m}{n}
\SetSymbolFont{MnLargeSymbols}{bold}{OMX}{MnSymbolE}{b}{n}
\DeclareFontShape{OMX}{MnSymbolE}{m}{n}{
    <-6>  MnSymbolE5
   <6-7>  MnSymbolE6
   <7-8>  MnSymbolE7
   <8-9>  MnSymbolE8
   <9-10> MnSymbolE9
  <10-12> MnSymbolE10
  <12->   MnSymbolE12
}{}
\DeclareFontShape{OMX}{MnSymbolE}{b}{n}{
    <-6>  MnSymbolE-Bold5
   <6-7>  MnSymbolE-Bold6
   <7-8>  MnSymbolE-Bold7
   <8-9>  MnSymbolE-Bold8
   <9-10> MnSymbolE-Bold9
  <10-12> MnSymbolE-Bold10
  <12->   MnSymbolE-Bold12
}{}

\let\llangle\@undefined
\let\rrangle\@undefined
\DeclareMathDelimiter{\llangle}{\mathopen}%
                     {MnLargeSymbols}{'164}{MnLargeSymbols}{'164}
\DeclareMathDelimiter{\rrangle}{\mathclose}%
                     {MnLargeSymbols}{'171}{MnLargeSymbols}{'171}
\makeatother

\newtheorem{thm}{Theorem}[section]

\newtheorem{lem}[thm]{Lemma}
\newtheorem{prop}[thm]{Proposition}

\theoremstyle{definition}
\newtheorem{rem}[thm]{Remark}
 
\newtheorem{dfn}[thm]{Definition} 
\newtheorem{exa}[thm]{Example} 
\newtheorem{remin}[thm]{Reminder} 
\newtheorem{notation}[thm]{Notation}

\newtheorem{thmx}{Theorem}
\begin{document} 

\title{Lipschitz rigidity for scalar curvature}

\author{Simone Cecchini}
\address{Mathematisches Institut,
Georg-August-Universit\"at, 
G{\"o}ttingen,
Germany}
\email{simone.cecchini@mathematik.uni-goettingen.de}

\author{Bernhard Hanke}
\address{Institut f\"ur Mathematik,
Universit\"at Augsburg, 
Augsburg,
Germany}
\email{bernhard.hanke@math.uni-augsburg.de}

\author{Thomas Schick}
\address{Mathematisches Institut,
Georg-August-Universit\"at, 
G{\"o}ttingen,
Germany}
\email{thomas.schick@math.uni-goettingen.de}

\begin{abstract}
%\noindent We prove the following Lipschitz rigidity result in scalar curvature geometry. 
Let $M$ be a closed smooth connected spin manifold of even dimension $n$, let $g$ be a Riemannian metric of regularity $W^{1,p}$, $p > n$, on $M$ whose distributional scalar curvature in the sense of Lee-LeFloch is bounded below by $n(n-1)$, and let  $f \colon (M,g) \to \S^n$ be a $1$-Lipschitz continuous (not necessarily smooth) map of non-zero degree to the unit  $n$-sphere. 
Then $f$ is a metric isometry. 
This generalizes a result of Llarull (1998) and answers in the affirmative a question of Gromov (2019) in his \emph{Four lectures}. 

Our proof is based on spectral properties of Dirac operators for  low regularity Riemannian metrics and twisted with Lipschitz bundles. 
We argue that the existence of a non-zero harmonic spinor field forces $f$  to be quasiregular in the sense of Reshetnyak, and in this way connect the powerful theory for quasiregular maps to the Atiyah-Singer index theorem.

\end{abstract}



\maketitle

%\tableofcontents

%\listoftodos
\section{Introduction}

Extremality and rigidity properties of Riemannian manifolds with lower scalar curvature bounds have played a major role in differential geometry for many years. 
We refer to  Gromov's {\em Four lectures on scalar curvature} \cite{Gromov4} for a comprehensive overview of the subject.
The following fundamental result of Llarull  illustrates the striking interplay of metric, curvature and homological information in this context. 
Throughout this paper we denote by $\S^n$ the unit $n$-sphere equipped with its standard Riemannian metric and induced Riemannian distance.

\begin{thm}[{\cite[Theorem B]{Lla98}}] \label{Llarull} 
Let  $M$ be a closed  smooth connected oriented  manifold of dimension $n \geq 2$ which admits a spin structure and is  equipped with a smooth Riemannian metric $g$ of scalar curvature $\scal_g \geq n(n-1)$. 
Furthermore, let  $f\colon (M,d_g)\to  \S^n$ be a smooth $1$-Lipschitz map of non-zero degree. 

Then $f$ is a Riemannian isometry. 
\end{thm}

Here $d_g$ denotes the Riemannian distance on $M$ induced by $g$. 
For $M$ equal to the unit $n$-sphere and $f = \id_M$, Theorem \ref{Llarull}  says that the round metric on $\S^n$ cannot be dominated by a  smooth Riemannian metric $g$ with  $\scal_g  \geq n(n-1)$, except by the round metric itself, see Llarull \cite[Theorem A]{Lla98}.

The proof of Theorem \ref{Llarull} is based on spectral properties of twisted Dirac operators and emphasizes  the fruitful interplay of spin and scalar curvature geometry. 
Indeed, it is unknown whether the spin condition is dispensable in  any dimension.

Goette and Semmelmann proved in \cite[Theorem 2.4]{GS02} and \cite[Theorem 0.1]{GS01} generalisations of Theorem \ref{Llarull} where $\S^n$ is  replaced by Riemannian manifolds with non-negative curvature operators and non-vanishing Euler characteristics, and  by closed K\"ahler manifolds of positive Ricci curvature, respectively. 
Lott \cite[Theorem 1.1]{Lott} proved similar extremality and rigidity results for smooth Riemannian manifolds with boundary. 
 
In the present work we will generalise Theorem \ref{Llarull} to Riemannian metrics $g$ and comparison maps $f$ with regularity less than $C^1$, thus highlighting the metric content of Llarull's theorem. 
This is close in spirit to other recent results, such as the conservation of lower scalar curvature bounds under $C^0$-convergence of smooth Riemannian metrics by Gromov \cite{Gro14} and Bamler \cite{Bam16},  the definition of lower scalar curvature bounds for $C^0$-Riemannian metrics via regularising Ricci flows by Burkhardt-Guim \cite{BG19},  and the positive mass theorem under low regularity assumptions by Lee-LeFloch \cite{LL15}.

\begin{remin} \label{reminder} Let $M$ be a connected smooth manifold and let $g$ be a continuous Riemannian metric on $M$, that is, $g \in C^0(M, T^*M \otimes T^*M)$ such that $g_x$ is a scalar product on $T_x M$ for $x \in M$. 
Given an absolutely continuous curve  $\gamma \colon [0,1] \to M$, the {\em length} of $\gamma$ is then defined by the formula 
 \[
   \ell( \gamma) :=   \int_{0}^1 | \gamma'(t) |_g  \, dt , 
\]
using that $|\gamma'|_g \in L^1([0,1] , \R)$, compare \cite[Prop. 3.7]{Burt15}. 

Thus $g$ induces a path metric $d_g \colon M \times M \to \R_{\geq 0}$, called the {\em Riemannian distance}, 
\[
   d_g(x,y) =  \inf \{ \ell ( \gamma) \mid \gamma \colon [0,1] \to M \text{ absolutely continuous}, \gamma(0)  = x, \gamma(1) = y\} . 
\]
The metric $d_g$ induces the manifold topology on $M$.
One obtains the same  distance function $d_g$ if the  infimum  is taken only over all regular smooth curves $\gamma \colon [0,1] \to M$ from $x$ to $y$. 

For more information about metric properties of smooth manifolds equipped with continuous Riemannian metrics, see Burtscher \cite[Section 4]{Burt15}. 
\end{remin}

\begin{dfn} \label{admissible} Let $M$ be a smooth manifold. 
A  continuous Riemannian metric $g$ on $M$ is called {\em admissible} if there exists some $p > n$ with 
\[
    g \in \Wloc^{1,p}( M , T^*M \otimes T^*M) .
\]
\end{dfn}

Recall that,  by the Sobolev embedding theorem, each section in $\Wloc^{1,p}(M, T^*M \otimes T^*M)$, $p > n$, has a unique continuous representative. 
We remark that Lipschitz continuous Riemannian metrics, which often arise from smooth Riemannian metrics in geometric gluing constructions as in Theorem \ref{theo:disk} below, are admissible. 

For an admissible Riemannian metric  $g$ on $M$  the scalar curvature  is defined as a distribution $\scal_g \colon C^{\infty}_c(M) \to \R$, see  Lee-LeFloch \cite{LL15} and Section \ref{sec:distrib_scal} below. 
In particular, we can define lower scalar curvature bounds in the distributional sense, see Definition \ref{lowerbounds}.

\begin{remin}  \label{differential} 
Let $M$ and $N$ be smooth manifolds with continuous Riemannian metrics $g$ and $h$. 
Let $f\colon (M, d_g)  \to (N,d_h)$ be Lipschitz continuous. 
If $f$ is (totally) differentiable at $x \in M$, we denote by 
\[
   d_x f\colon T_x M \to T_{f(x)} N 
\]
the differential of $f$ at $x$.
By Rademacher's theorem, $f$ is differentiable almost everywhere on $M$ with differential of regularity $\Lloc^{\infty}$.
\end{remin}

For $n \geq 3$ it suffices to assume  in Theorem \ref{Llarull}  that the smooth map $f$ is not $1$-Lipschitz, but only {\em area-nonincreasing}, {i.e.}, for all $x \in M$ the induced map  $\Lambda^2 d_xf \colon \Lambda^2 T_xM \to \Lambda^2 T_x \S^n$ is norm-bounded by $1$, see \cite[Theorem C]{Lla98}.
This notion is  generalized to Lipschitz maps as follows.

\begin{dfn}\label{def:weak_iso}
In the situation of Reminder \ref{differential}, we say that $df$   is  {\em area-nonincreasing a.e.}, if for almost all  $x \in M$ where $f$ is differentiable  the operator norm of the induced map 
\[
    \Lambda^2 d_x f \colon \Lambda^2 T_x M \to \Lambda^2 T_{f(x)} N 
\]
on the second exterior power of $T_x M$ satisfies $| \Lambda^2 d_x f| \leq 1$. 

\end{dfn}

Each $1$-Lipschitz map $f \colon (M,d_g)  \to (N,d_h)$ satisfies  $|d_x f| \leq 1$ for all $x \in M$ where $f$ is differentiable, see Proposition \ref{Lipschitzlength}. 
In particular, $df$  is  area-nonincreasing a.e.

We now state our main result.  
It provides an affirmative answer to a question  posed by Gromov in \cite[Section 4.5, Question (b)]{Gromov4}. 

\begin{thmx}\label{T:main_even}
Let $M$ be a closed smooth connected  oriented manifold of even dimension $n$ which admits a spin structure, let $g$ be an admissible  Riemannian metric with $\scal_g \geq n(n-1)$, and let  $f \colon (M,d_g) \to \S^n$ be a Lipschitz continuous map of non-zero degree. 
Futhermore, if $n = 2$, we assume that $f$ is $1$-Lipschitz, and if $n \geq 4$, then we assume that $df$ is area-nonincreasing a.e.

Then  $f$ is a metric isometry. 
\end{thmx}
 
\begin{rem} \label{weiter} \begin{enumerate}[(a)] 
  \item \label{regular}  If $g$ is assumed to be smooth then each metric isometry $(M, d_g) \to \S^n$ is a smooth Riemannian isometry by the Myers-Steenrod theorem. 
  \item  By Proposition \ref{Holder} each metric isometry between smooth manifolds equipped with admissible Riemannian metrics preserves the scalar curvature distributions. 
  In particular, in Theorem \ref{T:main_even}, we get $\scal_g \equiv n(n-1)$. 
  This will also follow from the proof of Theorem \ref{T:main_even} in Section \ref{sec:proofs}. 
   \item In the first version of this paper, we conjectured that Theorem \ref{T:main_even} also holds for all odd $n \geq 3$. 
   In the meantime  the preprint \cite{ML22}  by Lee-Tam appeared which develops an alternative approach to Theorem \ref{T:main_even} based on the Ricci and harmonic map heat flows. 
   In this way they generalize Theorem \ref{T:main_even} to all $n \geq 2$ under the assumption that $f$ is $1$-Lipschitz. 
   So far, however, the case of area-nonincreasing a.e.~maps $f$   seems to be
   accessible only via the Dirac operator method and remains open for odd $n
   \geq 3$. It is work in progress by the authors to generalise Theorem
   \ref{T:main_even} to these cases.
\end{enumerate} 
\end{rem}

Our approach  provides a new perspective on  scalar curvature rigidity results for smooth Riemannian manifolds with boundary, previously obtained by studying boundary value problems for Dirac-type operators  as in B\"ar-Hanke \cite[Theorem 2.19]{BH20}, Cecchini-Zeidler \cite[Corollary 1.17]{CZ21}, and Lott \cite[Corollary 1.2]{Lott}. 
This is illustrated by  the following strong comparison principle ``larger than hemispheres''.
We denote  by $\D^n_{\pm}$ the closed upper, respectively lower hemispheres of $\S^n$, equipped with the Riemannian metrics induced from $\S^n$. 
Furthermore, we identify $\S^{n-1}$ with the equator sphere of $\S^n$.

\begin{thmx}\label{theo:disk}
Let $(M,g)$ be a compact smooth connected oriented Riemannian manifold of even dimension $n$ which admits a spin structure and has boundary $\partial M$, and let $f \colon (M,d_g) \to \S^n$ be a Lipschitz continuous map such that 
\begin{enumerate}[\myicon] 
\item if $n = 2$, then $f$ is $1$-Lipschitz,
\item if $n \geq 4$, then $df$ is area-nonincreasing a.e., 
\item $\scal_g \geq n(n-1)$, 
\item $g$ has non-negative mean curvature along $\partial M$ with respect to the interior normal\footnote{With this convention the unit ball in $\R^n$ has mean curvature $n-1$.}, 
  \item $f( \partial M) \subset \D^n_{-}$ and $f \colon (M, \partial  M) \to (\S^n, \D^n_-)$ is of non-zero degree. 
\end{enumerate} 
Then $ \im ( f) = \D^n_+$ and $f \colon (M, g)  \to \D^n_+$ is a smooth Riemannian  isometry.
 \end{thmx}
We derive this result from Theorem \ref{T:main_even} by a doubling procedure in Section \ref{sec:proofs}.
Contrary to  Gromov \cite{Gro18, Gromov4}, we work directly with the resulting non-smooth metric on the double, thus dispensing with ``smoothing the corners''.

If  $f$ is smooth, $f(\partial M) \subset \S^{n-1}$ and  $f|_{\partial M}$ is $1$-Lipschitz with respect to $d_g|_{\partial M}$, then Theorem \ref{theo:disk}  follows from Lott \cite[Corollary 1.2]{Lott}.

\begin{rem} Our methods can also be applied to the more general situation treated  in \cite{GS02}. 
This issue is not discussed further in our paper. 
\end{rem}

Our paper is structured as follows. 
In Section \ref{rigess} we provide an infinitesimal characterisation of  metric isometries between smooth manifolds equipped with continuous Riemannian metrics.
This discussion, which is of independent interest,  makes essential use of Reshetnyak's  theory of quasiregular maps; see Rickman \cite{Rick93} for a comprehensive introduction. 

In Section \ref{sec:distrib_scal} we introduce the scalar curvature
distribution of an admissible Riemannian metric, following Lee-LeFloch  \cite{LL15}, and derive some properties relevant for our discussion. 

In Section \ref{S:Lipschitz bundles} we construct the spinor Dirac operator
twisted with  Lipschitz bundles on a spin manifold equipped with an admissible Riemannian metric. 
Relying on previous work of Bartnik-Chru\'sciel \cite{BartnikChrusciel}, we establish its main functional analytic properties  and prove an index formula in Theorem \ref{thm:index}.

The main result of Section \ref{sec:Lichnerowicz} is Theorem \ref{T:integral Lichnerowicz} which provides  a Schr\"odinger-Lichnerowicz formula for the twisted Dirac operator  considered in Section \ref{S:Lipschitz bundles} if the twist bundle is the pull back of a smooth bundle along a Lipschitz map.

Theorems \ref{T:main_even} and \ref{theo:disk} are  proved in Section \ref{sec:proofs}. 
Similar to the original proof of Theorem \ref{Llarull}, the Schr\"odinger-Lichnerowicz formula  is applied to a non-zero harmonic spinor field on $M$  -- whose existence is guaranteed by the index formula --  which implies  that the differential of $f$ is an isometry almost everywhere. 
However, in the absence of an inverse function theorem for Lipschitz maps, this no longer implies that $f$ is bijective, let alone a metric isometry. 
In fact $f$ could have folds, see Example \ref{folding}. 

Here our new observation is that the existence of this spinor field further implies that the differential of $f \colon M \to \S^n$ is {\em orientation preserving almost everywhere}, possibly after reversing the orientation of $M$. 
As explained in Section \ref{rigess}, this implies that $f$ is a quasiregular  map without branch points, hence a homeomorphism which must be a metric isometry by a curve length comparison argument. 

\smallskip 

{\it Acknowledgement.} We are grateful to Pekka Pankka for helpful correspondence about quasiregular maps, and to Christian B\"ar and Lukas Sch\"onlinner for useful comments. 
The authors were supported by SPP 2026 ``Geometry at Infinity'' funded by the Deutsche Forschungsgemeinschaft.

\section{A characterisation of metric isometries} \label{rigess} 

Let $M$ and $N$ be connected smooth manifolds, equipped with continuous Riemannian metrics $g$ and $h$.
Throughout this section, let
  \begin{equation}
   f \colon (M,d_g)    \to  (N,d_h)  \label{eq:setup_map}
 \end{equation}
be a locally Lipschitz continuous map. 

 \begin{prop}  \label{Lipschitzlength}  If $f$ is $L$-Lipschitz, then for all $x \in M$ where $f$ is differentiable we have $|d_x f| \leq L$.
  \end{prop} 

\begin{proof} Let $x \in M$ where $f$ is differentiable and let  $v \in T_x M$ with $|v|_{g_x} = 1$. 
Let $\gamma \colon (-\eps, \eps) \to M$ be a smooth curve with $\gamma(0) = x$ and $\gamma'(0) = v$.
The curve $f \circ \gamma \colon (-\eps, \eps) \to N$ is Lipschitz continuous, hence absolutely continuous. 
Furthermore it is differentiable at $t = 0$ by our choice of $x$. 

Let $\eta > 0$. 
Since $f$ is $L$-Lipschitz, since $|v|_{g_x} = 1$ and both $g$ and $\gamma'$ are continuous maps, we can assume (at least for some smaller $\eps$)  that  for $\delta \in (-\eps, \eps)$ we get 
\begin{equation} \label{absch} 
     d_{h}   ( f \circ \gamma(\delta) , f \circ \gamma(0))  \leq L \cdot d_g ( \gamma(\delta), \gamma(0)) \leq L \int_{0}^{\delta} | \gamma'(t)|_{g_{\gamma(t)}}  dt \leq (L + \eta)  \cdot | \delta|. 
\end{equation} 
By  \cite[Prop. 4.10]{Burt15} we have
\[
    \lim_{\delta \to 0 } \frac{d_{h} ( f \circ \gamma(\delta) , f \circ \gamma(0))  }{ |\delta| } = | (f \circ \gamma)'(0)|_{h}  
\]
and together with \eqref{absch}  this shows 
\[
    | (f \circ \gamma)'(0)|_{h} \leq L + \eta. 
\]
Letting $\eta$ go to $0$, we obtain $| (f \circ \gamma)'(0)|_{h} \leq L$, finishing  the proof.
\end{proof} 

If $g$ and $h$ are smooth, $f$ is proper (i.e., preimages of compact sets are compact), and $f$ induces a surjective map $\pi_1(M) \to \pi_1(N)$, then $f$ is a metric isometry if and only if it is a Riemannian isometry, {i.e.}, the differential $d_xf \colon (T_xM, g_x) \to (T_{f(x)} N,h_{f(x)} )$ is an isometry for all $x \in M$. 
In this section we provide  a similar characterisation  for continuous $g$ and $h$. 

\begin{dfn} 
We say that $df$ is an
\begin{enumerate}[\myicon] 
 \item {\em isometry a.e.}, if for almost all $x \in M$ where $f$ is differentiable, the differential $d_x f \colon (T_x M, g_x)  \to ( T_{f(x)} N, h_{f(x)}) $ is an isometry, 
 \item {\em locally orientation preserving isometry a.e.}, if for each $p \in M$ there exist oriented neighborhoods $p \in U \subset M$ and $f(p) \in V \subset N$ such that for almost all $x \in U$ where $f$ is differentiable, the differential $d_x f \colon (T_x M , g_x) \to  ( T_{f(x)}N, h_{f(x)}) $ is an orientation preserving isometry. 
 \end{enumerate} 
\end{dfn}

\begin{prop}  \label{Lipschitzlength_neu}  If $f$ is a metric isometry, then the differential  $df$ is a locally orientation preserving isometry a.e. 
  \end{prop} 

\begin{proof}  If follows from Proposition \ref{Lipschitzlength}  that $df$ is an isometry  a.e. 
Furthermore, the local homological mapping degree of the homeomorphism $f \colon M \to N$ is locally constant and hence  the differential  $d_x f$ is a locally orientation preserving isometry a.e. 
\end{proof}

The following result provides a converse.

\begin{thm} \label{essrigid} 
Let $f$ be proper, let $f$ induce a surjective map $\pi_1(M) \to \pi_1(N)$ and  let $df$ be a locally  orientation preserving isometry a.e. 

Then  $f$ is a metric isometry. 
\end{thm} 

\begin{rem} For smooth $g$ and $h$ this is implied by \cite[Thm.~1.1]{KMS19}. 
 \cite[Question 3 on p.~374]{KMS19}  asks  whether this result generalizes to metrics of regularity less than $C^{1, \alpha}$. 
Our Theorem \ref{essrigid} gives a positive answer for continuous $g$ and $h$. 
\end{rem}

\begin{exa} \label{folding} The $1$-Lipschitz map $\rho \colon \S^n \to \S^n$ with 
\[
    \rho(x^0, x^1, \ldots, x^n) =  (- |x^0|, x^1, \ldots, x^n)
\]
which leaves $\D^n_-$ pointwise fixed and reflects $\D^n_+$ onto $\D^n_-$ is not a metric isometry. 
Indeed, the differential $d\rho$ is an isometry a.e., but not a locally  orientation preserving isometry a.e. 

\end{exa}

The proof of Theorem \ref{essrigid} relies on the theory of quasiregular maps. 
We recall the basic definition from \cite{Rick93}, restricted to our setting.

\begin{dfn} \label{quasiregular} Let $n \geq 2$ and let $G \subset \R^n$ be a domain, that is, $G$ is a non-empty, open and connected subset of $\R^n$. 
Let $f \in \Liploc(G, \R^n)$ be a locally Lipschitz map. 

Let $K \geq 1$. 
The map $f$  is called {\em $K$-quasiregular} if  for almost all $x \in G$ where  $f$ is differentiable, we have
  \[
        | d_x f|^n \leq K \det (d_x f) . 
  \]
The map $f$ is called {\em quasiregular} if $f$ is $K$-quasiregular for some $K \geq 1$. 
\end{dfn} 

\begin{rem} Since locally Lipschitz maps are continuous and of regularity $\Wloc^{1,\infty}$, it follows  from  \cite[Proposition I.1.2]{Rick93} that $K$-quasiregular locally Lipschitz maps in the sense of Definition  \ref{quasiregular} are $K^{n-1}$-quasiregular in the sense of  \cite[Definition I.2.1 and the subsequent discussion]{Rick93}.

Note that for quasiregular  $f$, the Jacobian $\det(df)$ is non-negative almost everywhere on $G$. 
\end{rem} 

The following fact is implied by \cite[Theorem VI.8.14]{Rick93}. 

\begin{prop} \label{localhomeo} For all $n \geq 3$ there exists $\eps = \eps(n) > 0$ with the following property. 
Let $G \subset \R^n$ be a domain and let $f \in \Liploc (G, \R^n)$ be $(1+ \eps)$-quasiregular and nonconstant. 

Then  $f$ is a local homeomorphism, that is, for all $x \in G$ there exists an open neighborhood $x \in U \subset G$ such that $f(U) \subset \R^n$ is open and $f|_{U} \colon U \to f(U)$ is a homeomorphism. 
\end{prop} 

\begin{rem} Since each complex analytic map $\R^2 \supset G \to \R^2$ is $1$-quasiregular, the requirement $n \geq 3$ is necessary. 
\end{rem}

\begin{lem} \label{linalg} 
Let $\gamma_1$ and $\gamma_2$ be scalar products on $\R^n$. 
Assume that there exists $0 < \delta \leq \frac{1}{2}$ such that for all $v \in \R^n$ and $i = 1, 2$ we have
\[ 
      (1- \delta) |v| \leq |v|_{\gamma_i} \leq (1+ \delta) |v| . 
 \]
 Then, for all linear isometries $A \colon (\R^n, \gamma_1)  \to (\R^n, \gamma_2)$, we have 
\[
     |A|^n \leq (1+ 4\cdot 3^{2n}  \cdot \delta)   |\det A| . 
\]
\end{lem} 

\begin{proof} First, 
\[
 |A| = \max_{|v| = 1} |Av| \leq (1+\delta)  \max_{|v|_{\gamma_1} =1}  |Av| \leq \frac{1+\delta}{1-\delta}  \max_{|v|_{\gamma_1}  = 1} |Av|_{\gamma_2} =  \frac{1+\delta}{1-\delta} . 
\]
Second, by a singular value decomposition of $A$, there exist orthonormal bases $(e_1, \ldots, e_n)$ and $(\eps_1, \ldots, \eps_n)$ of $\R^n$ together with real numbers $\mu_i > 0$ so  that 
\[
      A e_i = \mu_i \eps_i, \qquad  1 \leq i \leq n . 
\]
The numbers $\mu_i$ satisfy 
\[
    \mu_i = |Ae_i| = |e_i|_{\gamma_1} \cdot  \big| A \frac{e_i}{|e_i|_{\gamma_1}} \big| \geq \frac{1-\delta}{1+\delta} \cdot  \big| A \frac{e_i}{|e_i|_{\gamma_1}} \big|_{\gamma_2} =  \frac{1-\delta}{1+\delta}  
\]
so that 
\[
    |\det A| = \prod_{i=1}^n \mu_i \geq \left(  \frac{1-\delta}{1+\delta} \right)^n . 
\]
Altogether, using $\delta \leq 1/2$, 
\begin{align*} 
    |A|^n&  \leq \left(  \frac{1+\delta}{1-\delta} \right)^{n} \leq  \left(  \frac{1+\delta}{1-\delta} \right)^{2n} |\det A| \\
            & = \left( 1 +  \frac{ 2 \delta}{1-\delta} \right)^{2n} |\det A| \leq  ( 1 + 4 \delta)^{2n} |\det A|. 
\end{align*} 
The claim now follows from
\[
     ( 1 + 4\delta)^{2n} = \sum_{j=0}^{2n} \binom{2n}{j} (4 \delta)^j = 1 + 4
     \delta \sum_{j=1}^{2n} \binom{2n}{j} \underbrace{(4 \delta)^{j-1}}_{\le
       2^j}\cdot 1^{2n-j} \leq 1 + 4 \delta ( 2+1)^{2n} . 
\]
\end{proof} 

The following fact is close to \cite[Lemma 4.2]{LS16}. 

\begin{lem} \label{goodcharts} Let $n=\dim M \geq 2$ and $df$ be a  locally orientation preserving isometry a.e. 

Then, for each $\eps > 0$ and each $x \in M$, there exist local charts $(U,\phi)$ around $x$ and $(V,\psi)$ around $f(x) \in N$ such that $f(U) \subset V$ and the induced map  
\[
    \hat f = \psi \circ f \circ \phi^{-1} \colon \phi(U) \to \psi(V) \subset \R^n 
\]
is a  $(1+\eps)$-quasiregular Lipschitz map. 
\end{lem} 

\begin{proof} 
By assumption made in \eqref{eq:setup_map}, we also have $\dim N = n$. 
Choose connected smooth local oriented charts $(U, \phi)$ around $x \in M$ and
$(V, \psi)$ around $y : = f(x) \in N$ with $f(U) \subset V$, and let $\hat f =
\psi \circ f \circ \phi^{-1} \colon \phi(U) \to \psi(V) \subset \R^n$ be the
induced map which we can assume to be Lipschitz by convention~\eqref{eq:setup_map}. 
%Set $\hat x := \phi(x)$ and $\hat y := \psi(y)$. 

Let $\hat g = (\phi^{-1})^*(g) $ and $\hat h = (\psi^{-1})^*( h)$ be the induced continuous Riemannian metrics on $\phi(U)\subset \R^n$ and $\psi(V)\subset \R^n$. 
We may assume that for almost all $\zeta \in \phi(U)$ where $\hat f$ is differentiable, the differential $d_\zeta \hat f \colon (\R^n,  \hat g_\zeta) \to (\R^n , \hat h_{\hat f(\zeta)})$ is an orientation preserving isometry. 

By composing $\phi$ and $\psi$ with linear  orientation preserving isomorphisms of $\R^n$, we may assume that $\hat g_{ \phi(x)  }$ and $\hat h_{\psi(y) }$ are equal to the  standard Euclidean scalar product. 
Since $\hat g$ and $\hat h$ are continuous,  this implies, possibly after passing to smaller $U$ and $V$, that with $\delta := \frac{\eps}{2^{2(n+1)}}$ we get  that
\begin{align*} 
     (1- \delta) |v| \leq |v|_{\hat g_\zeta } &  \leq (1+ \delta) |v| , \qquad
                                                \forall \zeta \in \phi(U),\;
                                                \forall v \in \R^n ; \\
     (1 - \delta) |v| \leq |v|_{\hat h_\xi } & \leq (1+ \delta) |v| , \qquad
                                               \forall \xi \in \psi(V), \; \forall v \in \R^n. 
\end{align*} 
By Lemma \ref{linalg} the assertion follows. 
\end{proof}

\begin{prop} \label{homeom} The map $f$ in Theorem \ref{essrigid} is a homeomorphism
\end{prop} 

\begin{proof} For $n = \dim( M) \geq 3$,  Proposition \ref{localhomeo} and  Lemma \ref{goodcharts}  imply that $f$ is a local homeomorphism and hence a covering map since $f$ is proper. 
Since $f$ induces a surjective map $\pi_1(M) \to \pi_1(N)$,  this shows that $f$ is a homeomorphism.

For $n = 2$ we consider the Riemannian products $(M \times \R, g + dt^2)$ and $(N \times \R, h + dt^2)$ where $dt^2$ is the standard Riemannian metric on $\R$ and observe that the differential of the locally Lipschitz map 
\[
   f \times \id \colon (M \times \R , d_{g +  dt^2}) \to (N \times \R, d_{h +  dt^2})
\]
is a locally orientation preserving isometry a.e. 
Hence $f \times \id$ is a homeomorphism and the same then holds for $f$. 
For $n = 1$ a similar trick applies. 
\end{proof}

With Proposition \ref{homeom},  the proof of Theorem \ref{essrigid} is
completed by the following.

\begin{prop} \label{isomhomeo} 
Let $f$ be a homeomorphism and let $df$ be an isometry a.e. 

Then $f$ is a metric isometry. 
\end{prop}

\begin{proof} 
We need to show that for all $x, y \in M$ we have 
 \begin{align} 
   \label{lessthan} d_g(x, y) & \leq d_{h}(f(x) ,f(y)) , \\ 
  \label{largerthan}  d_g(x,y) & \geq d_{h}(f(x), f(y)) . 
\end{align} 
Since the local homological mapping degree of $f$ is locally constant, we may assume that the differential $df$ is an orientation preserving isometry a.e. 
Then 
\[
    \Mr := \{ x \in M \mid  \text{f differentiable at $x$ and $d_x f$ is orientation preserving isometry}\} 
\]
is a subset of $M$ of full measure and contained in the subset of $M$ where $f$ is differentiable. 
In particular, since $f$ is Lipschitz, the image  $f(M \setminus \Mr) \subset N$ has measure zero, compare  \cite[Lemma  4.1]{Pou77}.

Let $p \in \Mr$ and set $q=f(p)\in N$. 
Then the map $f^{-1}$ is differentiable at $q$ with differential $(d_p f)^{-1}$. 
 It follows that $f^{-1} \colon N \to M$ is almost everywhere differentiable and at almost all $q \in N$ where $f^{-1}$ is differentiable its differential $d_q f^{-1}$ is an isometry. 
Note that we cannot assume at this point that $f^{-1}$ is Lipschitz. 

Given $\eta > 0$, there exists a regular smooth curve $\gamma \colon [0,1] \to N$ joining $f(x)$ and $f(y)$ for which 
 \begin{equation*} \label{aprox} 
      \ell_{h}(\gamma) \leq d_{h}(f(x) ,f(y) ) + \eta. 
 \end{equation*} 
We claim that 
\begin{equation} \label{preimage} 
    d_g( x, y ) \leq \ell_{h} (\gamma). 
\end{equation} 
This implies \eqref{lessthan} by letting $\eta$ go to zero. 

To show \eqref{preimage} we may (after subdividing $\gamma$) assume that  $\gamma \colon [0,1] \to N$ is an embedded smooth curve and that there exist local charts $(U,\phi)$ on $M$ and $(V, \psi)$ on $N$ such that $\im (\gamma) \subset V$ and 
\[
     \psi \circ \gamma (t) = (t, 0, \ldots, 0) . 
\]
Furthermore, by Proposition \ref{goodcharts}  we may assume that $\hat f \colon \phi(U) \to \psi(V)$ is quasiregular. 
By \cite[Corollary II.6.5]{Rick93} the inverse $\hat f^{-1} \colon \psi(V) \to \phi(U)$ is also quasiregular. 

Choose $\delta > 0$ such that $[0,1]  \times (-\delta, \delta)^{n-1} \subset \psi(V)$ and for $\tau \in 0 \times (-\delta, \delta)^{n-1}$ and $t \in [0,1]$ set 
\[
  \gamma_{\tau} (t) = \psi^{-1} \left( (t, 0, \ldots, 0)  + \tau\right) \in V . 
\] 
Note that  $\gamma_0 = \gamma$. 
Since  $\hat f^{-1} \colon \psi(V) \to \phi(U)$ is quasiregular, it is absolutely continuous on lines, see \cite[page 5]{Rick93}, and hence $f^{-1} \circ  \gamma_\tau  \colon [0,1] \to U \subset M$ is absolutely continuous for almost all $\tau$. 
Since $d_g \colon M \times M \to \R_{\geq 0}$ is continuous and $\ell_{h}(\gamma_\tau)$ depends continuously on $\tau$, we may therefore assume for showing \eqref{preimage} that  $ f^{-1}  \circ \gamma \colon [0,1] \to U$ is absolutely continuous.

Furthermore, since $f(M \setminus \Mr) \subset N$ has measure zero, the intersection $\im (\gamma_\tau) \cap f(M \setminus \Mr)$ has measure zero for almost all $\tau$ by the Fubini theorem (for the $1$-dimensional Lebesgue measure on the $1$-dimensional submanifold $\im ( \gamma_\tau) \subset N$). 
By a similar continuity argument as before, we may therefore assume  that $\im (\gamma) \cap f(M \setminus \Mr)$ has measure zero. 
In particular, $d_q f^{-1}$ exists and is an isometry for almost all $q \in \im (\gamma)$. 

Under these assumptions, we obtain 
\[
 d_g( x , y)  \leq  \ell_g( f^{-1}  \circ \gamma)  = \int_0^1  | (f^{-1}  \circ \gamma)'(t) |_g dt = \int_0^1 | \gamma'(t)|_{h}  dt = \ell_{h} (\gamma) . 
\]
This verifies \eqref{preimage} and finishes the proof of \eqref{lessthan}. 

The proof of \eqref{largerthan} is analogous and in fact easier since $f$ is Lipschitz by assumption. 
Assume that  $\gamma \colon [0,1] \to M$ is an embedded smooth curve joining $x$ and $y$ and there are local charts  $(U,\phi)$ on $M$ and $(V, \psi)$ on $N$ such that $\im (\gamma) \subset U$ and 
\[
     \phi \circ \gamma (t) = (t, 0, \ldots, 0) . 
\]
It is enough to show that under theses circumstance we get 
\begin{equation} \label{distanceimage} 
     d_{h} (f(x) , f(y) ) \leq \ell_g(\gamma). 
\end{equation} 

Since $f$ is Lipschitz, $f \circ \gamma \colon [0,1] \to V$ is absolutely continuous. 
By a similar argument as for  \eqref{lessthan}, considering a family of curves $\gamma_\tau (t) = \phi^{-1}( (t, 0, \ldots, 0) + \tau)$, we may assume that for almost all $p\in \im ( \gamma) $ the map $f$ is differentiable at $p$ and $d_p f$ is an isometry. 
This implies \eqref{distanceimage} and completes the proof of \eqref{largerthan}.

\end{proof}

\begin{rem} Assume that $f$ is proper and a locally orientation preserving isometry a.e., but $f$ does not necessarily induce a surjective map $\pi_1(M) \to \pi_1(N)$. 
Let $\pi \colon \tilde N \to N$ be the universal covering of $N$ and let $\tilde f \colon \tilde M \to \tilde N$ be the pull back of $f$ along $\pi$, which is still proper. 
With the continuous Riemannian metric $\tilde g$ and $\tilde h$ induced from $g$ and $h$, respectively, the map $\tilde f$ is a locally orientation preserving isometry a.e. 

Theorem \ref{essrigid} implies that the map  $\tilde f \colon (\tilde M , d_{\tilde g}) \to (\tilde N, d_{\tilde h})$ is a metric isometry. 
Hence the original map $f \colon (M, d_g) \to (N, d_h)$ is a locally isometric covering map. 
\end{rem}

\section{Distributional scalar curvature}\label{sec:distrib_scal}
Let $M$ be a smooth manifold of dimension $n$. 
Given a real or complex Lipschitz vector bundle $E \to M$, we denote by $\Liploc(M,E)$ the vector space of locally Lipschitz continuous sections of $E$ and by $\Lloc^p(M,E)$ the vector space of 
\begin{enumerate}[\myicon]  
\item  locally $p$-integrable sections of $E$ if $1 \leq p < \infty$, 
\item  locally essentially bounded integrable sections of $E$  if $p = \infty$. 
\end{enumerate} 
These spaces are defined with respect to some, hence any, bundle metric on $E$ and continuous volume density on $M$.  

\begin{dfn} \label{connection} 
An \emph{affine connection} on $E$ is a  linear map
\[
 \nabla \colon   \Liploc(M,E)   \longrightarrow \Lloc^2(M, T^* M \otimes E)  
\]
which  for all scalar valued Lipschitz functions $f$ on $M$ satisfies the Leibniz rule 
\[
    \nabla ( f \eta) = df \otimes \eta + f \,  \nabla \eta. 
\]
Let $ \langle - ,  - \rangle$ be a Euclidean (if $E$ is real) or Hermitian (if
$E$ is complex) bundle metric on $E$. 
We say that a connection  $\nabla$ on $E$ is  \emph{metric} if, for all $\eta_1, \eta_2 \in \Liploc(M,E)$, we get 
\begin{equation} \label{metric} 
    d  \left<\eta_1,\eta_2\right>=\left< \nabla   \eta_1,\eta _2\right>+\left<\eta_1,\nabla  \eta_2\right> . 
\end{equation}
\end{dfn} 

Working in local Lipschitz frames of $E$, each affine connection $\nabla$ uniquely extends to a linear map $ \nabla \colon \Wloc^{1,2}(M,E) \to \Lloc^2(M, T^*M \otimes E)$ satisfying the  Leibniz rule. 
If the given connection $\nabla$ is metric, then \eqref{metric} continues to hold for $\eta_1, \eta_2 \in \Wloc^{1,2}(M,E)$. 
Furthermore, the usual partition of unity argument shows  that on each Lipschitz bundle $E \to M$ with bundle metric there exists a metric connection. 

\begin{exa} \label{Levi} 
Let $g$ be an admissible Riemannian metric of regularity $\Wloc^{1,p}$, $p > n$, on $M$. 
With respect to $C^1$-regular local coordinates $(x^1, \ldots, x^n)$ on   $U \subset M$, the metric $g$ has Christoffel symbols 
\[
   \Gamma^k_{ij} = \frac{1}{2} g^{k\ell} ( \partial_i g_{\ell j} + \partial_j g_{i\ell} - \partial_\ell g_{ij})  \in \Lloc^p(U) , \qquad 1 \leq i,j, k \leq n. 
\] 
Here, as usual, $(g^{k\ell})$ is the inverse of the metric tensor $(g_{ij})$, and we sum over the index $\ell$ according to the Einstein summation convention. 

Using the multiplication $\Liploc(U, \R) \times \Lloc^p(U) \to \Lloc^2(U)$  and the inclusion $\Lloc^{\infty}(U) \subset \Lloc^2(U)$, we obtain the  {\em Levi-Civita connection}
\[
   \nabla_g (f^j \partial_j) = dx^i \otimes  \big( \partial_i f^k  + f^j \Gamma_{ij}^k  \big) \partial_k 
\] 
which is metric with respect to $g$.

\end{exa} 

The scalar curvature of metrics of regularity less than $C^2$   cannot be defined in the usual way. 
However, it can be defined as a distribution as follows; compare \cite[Section 2.1]{LL15}.
Let $(x^1, \ldots, x^n)$ be smooth local coordinates on $U \subset M$. 
Recall that for $C^2$-regular $g$, the Riemannian curvature tensor on $U$ is written in terms of Christoffel symbols as 
\[
  \displaystyle {R_{ijk}}^{\ell}= \partial_j  \Gamma _{ik}^{\ell} - \partial_i \Gamma _{jk}^{\ell}+  \Gamma _{ik}^{p}\Gamma _{jp}^{\ell}-\Gamma _{jk}^{p}\Gamma _{ip}^{\ell}  .
\]
Hence the scalar curvature of $g$ has the local expression
\begin{equation} \label{scalar}
    \scal_g|_U  = \partial_k  V^k + F 
\end{equation} 
where $V^k$ and $F$ are the smooth functions on $U$ defined by 
\begin{align} 
    \label{V}   V^k  & \coloneqq    g^{ij} \Gamma^{k}_{ij} - g^{ik} \Gamma^{j}_{ji} ,  \\
     \label{F}   F  & \coloneqq    -  ( \partial_k g^{ij})  \, \Gamma^{k}_{ij} +(  \partial_k g^{ik} ) \, \Gamma^j_{ji} + g^{ij} \left( \Gamma^{k}_{k \ell} \Gamma^{\ell}_{ij} - \Gamma^k_{j \ell} \Gamma^{\ell}_{ik} \right) . 
\end{align} 
Note that in  \eqref{scalar} the second order derivatives of the metric tensor occur only linearly, while its first order derivatives occur quadratically.

Let $d\mu_g$ be the Riemannian volume element of $g$ which locally on $U$, we write in the form $d\mu_g = \sqrt{ \det(g_{ij})} dx$. 
 For $u \in C^{\infty}_c(U)$, we hence obtain, using integration by parts, 
\begin{equation} \label{defscal} 
    \int_M\scal_gu\,d\mu_g = \int_U \bigl(- V^k \, \partial_k \bigl( u \,  \sqrt{\det(g_{ij})}\bigr)  + F  \, u \,  \sqrt{\det (g_{ij})} \bigr) dx . 
\end{equation} 
In particular, the right hand integral is independent of  the choice of smooth local coordinates on $U$. 

Now let $g$ be admissible of regularity $\Wloc^{1,p}$, $p > n$.  
Then in \eqref{V} and \eqref{F} we have $V^k  \in \Lloc^{2}(U)$ and $F \in \Lloc^{1} (U)$. 
Furthermore, since $g$ is of regularity $\Wloc^{1,p}$, $\det(g_{ij}) > 0$ and $\sqrt{-} \colon (0,\infty) \to (0, \infty)$ is smooth with locally bounded derivative, we obtain $\sqrt{\det(g_{ij})} \in \Wloc^{1,p}(U)$ by the chain rule for Sobolev functions. 
Hence  the right hand side of \eqref{defscal} is still defined. 
Obviously, for fixed $u$ it is continuous in $g$ in the $\Wloc^{1,p}$-topology, which we will use frequently in the sequel. 
For example,  since any   $\Wloc^{1,p}$-regular metric can be approximated by $C^2$-regular metrics in the $\Wloc^{1,p}$-topology, the right hand side of \eqref{defscal} is independent of the choice of smooth local coordinates.

By a partition of unity on $M$, we hence obtain a linear functional
\begin{equation*}
\scal_g \colon C^{\infty}_c(M) \to \R
\end{equation*}
 which is uniquely determined by 
 \begin{equation*}\label{eq:distributional_scal_g}
    \llangle \scal_g,u\rrangle = \int_U \bigl(- V^k \, \partial_k \bigl( u \, \sqrt{\det(g_{ij})}\bigr)  + F \, u \, \sqrt{\det (g_{ij})} \bigr)  dx  
\end{equation*} 
 if $u$ is supported in a coordinate neighborhood $U \subset M$ with smooth  coordinates $(x^1, \ldots, x^n)$. 

\begin{dfn}  The functional $\scal_g \colon C_c^{\infty}(M) \to \R$ is called  the {\em scalar curvature distribution} of $g$. 
\end{dfn}

\begin{rem} Our definition of the distribution $\scal_g$ coincides with \cite[Definition 2.1]{LL15} which uses a smooth background metric on $M$ instead of local coordinates. 
\end{rem}

The definition of $\llangle \scal_g, u \rrangle$ easily extends to  Lipschitz functions $u \in \Lip_c(M, \CC)$ and then defines a distribution with  values in $\CC$. 
This implies that if $E \to M$ is a complex Lipschitz bundle with Hermitian inner product $\langle - , - \rangle$, we obtain a sesquilinear  functional 
\[
 \mathscr{S}_g \colon \Lip_c(M,E) \times \Lip_c(M,E) \to \CC, \qquad \mathscr{S}_g (  \eta_1, \eta_2) :=  \llangle \scal_g, \langle \eta_1, \eta_2 \rangle  \rrangle. 
\]

\begin{prop}\label{cor:scal_on_sections}
The  scalar curvature distribution  $\mathscr{S}_g$ extends to a continuous sesquilinear functional 
\[
   \mathscr{S}_g \colon \Wloc^{1,2}(M , E) \times \Wloc^{1,2} (M,E)  \to \CC . 
\]
\end{prop} 

\begin{proof} Let $g$ be of regularity $\Wloc^{1,p}$, $p > n$. 
Let $\eta , \theta \in \Wloc^{1,2}(M,E)$.
Then $\langle \eta , \theta \rangle  \in  \Lloc^{\frac{n}{n-2}}(M)$ for $n \geq 3$ and  $\langle \eta , \theta \rangle   \in \Lloc^{\frac{p}{p-2}}(M)$ for $n = 2$  by the Sobolev embedding theorem. 
Pick a metric connection $\nabla^E$ on $E$. 
We then obtain 
\[
  d  \langle \eta , \theta \rangle =\left<\nabla^E \eta , \theta \right> + \langle \eta , \nabla^E \theta \rangle  \in \Lloc^{\frac{n}{n-1}}(M  ,T^*M )
\]
by the H\"older inequality. 
  
Moreover, if $(\eta_i)$ and $(\theta_i)$ are  sequences  in $\Wloc^{1,2}(M,E)$ converging to $\eta$ and $\theta$ in $\Wloc^{1,2}(M,E)$, then  $\langle \eta_i , \theta_i \rangle  \to \langle \eta ,\theta \rangle$ in $\Lloc^{\frac{n}{n-2}}(M)$  for $n\geq 3$, respectively in $\Lloc^{\frac{p}{p-2}}(M)$ for $n =2$, and hence $d \langle \eta_i, \theta_i \rangle \to d \langle \eta, \theta \rangle$ in $\Lloc^{\frac{n}{n-1}}(M, T^* M)$ by a similar argument. 
 
 Writing the right hand integrand in \eqref{defscal} in the form 
 \[
  - V^k \cdot \partial_k u \cdot \sqrt{\det(g_{ij})}   - V^k \cdot u \cdot \partial_k \sqrt{\det(g_{ij})}   + F \cdot u  \cdot \sqrt{\det (g_{ij})}
 \]
for $u := \langle \eta, \theta \rangle$,  the claim now follows from the H\"older inequality. 
Indeed, since $\sqrt{\det(g_{ij})} \in \Wloc^{1,p}(U)$, we can argue as follows: 
 \begin{enumerate}[\myicon]
    \item $V^k \in \Lloc^{n}(U)$, $\partial_k u \in \Lloc^{\frac{n}{n-1}}(U)$, $\sqrt{\det(g_{ij})} \in C^0(U)$  and $\frac{1}{n} + \frac{n-1}{n} = 1$, 
    \item $V^k \in \Lloc^{n}(U)$, $u \in \Lloc^{\frac{n}{n-2}}(U)$, $\partial_k \sqrt{\det(g_{ij})} \in \Lloc^{n}(U)$  and $\frac{1}{n} + \frac{n-2}{n} + \frac{1}{n} = 1$ for $n \geq 3$,
    \item $V^k \in \Lloc^{p}(U)$, $u \in \Lloc^{\frac{p}{p-2}}(U)$, $\partial_k \sqrt{\det(g_{ij})} \in \Lloc^{p}(U)$  and $\frac{1}{p} + \frac{p-2}{p} + \frac{1}{p} = 1$ for $n = 2$,
        \item $F \in \Lloc^{n/2} (U)$, $u \in \Lloc^{\frac{n}{n-2}}(U)$, $\sqrt{\det(g_{ij})} \in C^0(U)$ and $\frac{2}{n} + \frac{n-2}{n} = 1$ for $n \geq 3$,
        \item $F \in \Lloc^{p/2}(U)$ , $u  \in \Lloc^{\frac{p}{p-2}}(U)$, $\sqrt{\det(g_{ij})} \in C^0(U)$ and $\frac{2}{p} + \frac{p-2}{p} = 1$ for $n=2$. 
\end{enumerate} 
\end{proof}

We finally show that the scalar curvature distribution is invariant under metric isometries.

\begin{prop} \label{Holder} Let $M$ and $N$ be smooth manifolds of dimension $n$ and equipped with admissible Riemannian metrics $g$ and $h$ of regularity $\Wloc^{1,p}$, $p > n$. 
Let $f \colon (M, d_g) \to (N, d_h)$ be a metric isometry. 

Then  $\scal_g = f^* \scal_h$ in the sense that, for all $u \in C^{\infty}_c(N)$, we have
\[
    \llangle \scal_g , u \circ f \rrangle = \llangle \scal_h , u \rrangle. 
\]
\end{prop}

\begin{proof} Working in local coordinate neighborhoods we may assume that $M$ and $N$ are open subsets of $\R^n$. 
Let $(u^1,\dots,u^n)$ be local harmonic coordinates on $\Omega' \subset N \subset \R^n$.
Since $h$ is $\Wloc^{1,p}$-regular, by \cite[Chapter 3, (9.40)]{Taylor} we have $u^i \in \Wloc^{2,p}(\Omega')$ for $1 \leq i \leq n$. 

Let $\Omega := f^{-1}(\Omega') \subset M \subset \R^n$. 
In order to simplify the notation we replace $M$ and $N$ by $\Omega$ and $\Omega'$. 

Since $f \colon \Omega \to \Omega'$ is a metric isometry, the functions $v^i:= u^i \circ f \colon \Omega \to \R$, $1 \leq i \leq n$,  are weakly harmonic, hence harmonic, for the $\Wloc^{1,p}$-regular metric $g$, by an argument as in   \cite[page 2417]{CH06}.
Again, this implies  $v^i \in \Wloc^{2,p}(\Omega)$ for $1 \leq i \leq n$. 

The Sobolev embedding theorem implies that $u = (u^1, \ldots, u^n)  \colon \Omega \to u(\Omega) \subset \R^n$ and $v = (v^1, \ldots, v^n) \colon \Omega' \to v(\Omega') \subset \R^n$ are $C^1$-diffeomorphisms. 
Furthermore, a short calculation shows that their inverses are of regularity $\Wloc^{2,p}$ and  that the same holds for 
\[
    f = u^{-1} \circ v \colon \Omega \to \Omega'. 
\]
By Proposition \ref{Lipschitzlength_neu}  we have $f^*h=g$. 
It remains to show that this and the $W^{2,p}$-regularity of $f$ imply that 
\begin{equation} \label{transformation} 
    \llangle \scal_g , u \circ f \rrangle = \llangle \scal_h , u \rrangle, \qquad u \in C^{\infty}_c( \Omega'). 
\end{equation} 

For proving this claim, let  $(h_{\nu})$ be a sequence of smooth Riemannian metrics on $\Omega'$ converging to $h$ in the $\Wloc^{1,p}$-topology. 
Since $f$ is $\Wloc^{2,p}$-regular, this implies that $g_{\nu} := f^* h_{\nu} \in \Wloc^{1,p}(\Omega, T^* \Omega \otimes T^*\Omega)$ converges to $g$ in the $\Wloc^{1,p}$-topology. 
For proving \eqref{transformation} we may therefore assume that $h$ is smooth.
This implies that $g = f^* h$ is still of regularity $\Wloc^{1,p}$. 

Next let $(f_{\nu})$   be a sequence of $C^{\infty}$-diffeomorphisms $\Omega \to \Omega'$ converging to $f$ in the $\Wloc^{2,p}$-topology. 
This implies that $g_{\nu} := f_\nu^* h$ converges to $g$ in the $\Wloc^{1,p}$-topology. 
For proving \eqref{transformation} we may hence assume that both $h$ and $f$,  and hence  $g = f^* h$, are smooth. 

But for smooth $g$, $f$ and $h$, the scalar curvatures $\scal_g \colon \Omega \to \R$ and $\scal_h \colon \Omega' \to \R$  are defined by the classical formula \eqref{scalar}, and furthermore $\scal_g = \scal_h \circ f$ by the invariance of the scalar curvature of smooth Riemannian metrics under smooth diffeomorphisms.
From this Equation \eqref{transformation} follows.
\end{proof}

\section{Dirac operators twisted with Lipschitz bundles}\label{S:Lipschitz bundles}

Let $M$ be a smooth oriented $n$-dimensional manifold which admits a spin structure,  i.e., the second Stiefel-Whitney class of $M$ vanishes. 
Let $\gamma$ be some smooth Riemannian metric on $M$, which remains fixed throughout. 
Since $M$ admits a spin structure, we may choose a  smooth
$\Spin(n)$-principal bundle $\Prin_{\Spin}(M,\gamma) \to M$ together with a
$2$-fold covering map to the $\SO(n)$-principal bundle  $P_{\SO}(M,\gamma)
\to M$ of smooth positively oriented $\gamma$-orthonormal frames which is
compatible with the canonical map $\Spin(n) \to \SO(n)$.

Let $W$ be a complex left $\cl_n$-module with $\Spin(n)$-invariant Hermitian inner product. 
We obtain the associated {\em spinor bundle}   
\[
  S = \Prin_{\Spin}(M,\gamma) \times_{\Spin(n)} W \to M .
\]
It carries a Hermitian inner product induced by the inner product on $W$ and a metric connection  $\nabla^\gamma$ induced by the Levi-Civita connection on $(M,\gamma)$. 
Furthermore, it carries  a skew-adjoint module structure over the Clifford algebra bundle $\cl(TM, \gamma) \to M$ with respect to which $\nabla^\gamma$ acts as a derivation. 
See~\cite[Ch.~II]{LM89} for details.

Let  $g$ be an admissible Riemannian metric on $M$ of regularity $\Wloc^{1,p}$, $p > n$. 
In order to avoid the discussion of non-smooth spinor bundles, we develop the spin geometry for $(M,g)$ entirely on the smooth bundle $S \to M$ similarly as in \cite[Section 3.1]{LL15}. 

We recall that there exists a unique $\gamma$-self-adjoint and positive linear isomorphism $B_g \colon TM \to TM$ covering $\id_M\colon M \to M$ and satisfying 
\[
     \gamma(v,w) = g (B_g(v), w) , \quad v, w \in T_x M, \, x \in M. 
\]
It has a unique positive square root $b_g \colon TM \to TM$ which is again $\gamma$-self-adjoint and satisfies 
\[
    \gamma(v,w) = g (b_g v, b_g w), \quad v, w \in T_x M, x \in M. 
\]
Let $S^2(T^*M) \to M$ denote the vector bundle of symmetric $(0,2)$-tensors over $M$, and let $S^2_+(T^*M) \subset S^2(T^*M)$ be the open subset of positive definite symmetric $(0,2)$-tensors. 

\begin{lem}  \label{regb} The assignment $g \mapsto b_g$ defines a continuous map 
\[
   b \colon  \Wloc^{1,p}(M, S^2_+(T^*M) ) \to \Wloc^{1,p} (M, \End(TM))
\]
where $p>n$.
In particular, each $b_g$ maps local smooth $\gamma$-orthonormal frames to local $W^{1,p}$-regular $g$-orthonormal frames.
\end{lem} 

\begin{proof} Let $\Sym  \subset \R^{n \times n}$ be the linear subspace of symmetric matrices and $\Sym^+ \subset \Sym$ be the open subset of positive definite symmetric matrices. 
The map $\Sym^+ \to \Sym^+$, $A \mapsto A^2$, is bijective and smooth with
differential which is everywhere invertible. 
It is hence a diffeomorphism with smooth inverse which we denote by $A \mapsto \sqrt{A}$. 

Denoting by ${\Sym}^+ ( TM) \to M$ the smooth bundle of positive definite $\gamma$-self-adjoint endomorphisms of $TM$, we obtain a continuous map 
\[ 
  \sqrt{-}  \colon \Wloc^{1,p}( M , {\Sym}^+ (TM)) \to \Wloc^{1,p}(M , {\Sym}^+(TM)), 
\]
induced by the map $A \mapsto \sqrt{A}$ in each fiber. 

The assignment $g \mapsto B_g$ induces a continuous map 
\[
   B \colon  \Wloc^{1,p}( M , S^2_+(T^*M) ) \to \Wloc^{1,p}(M , {\Sym}^+(TM)) . 
\]
Since $b  = \sqrt{B}$, this concludes the proof of Lemma \ref{regb}. 
\end{proof} 

\begin{notation} \label{kurz} For $v \in TM$ we set $v^g := b_g(v)$. 
\end{notation}

Let $(e_1, \ldots, e_n)$ be a smooth positively oriented $\gamma$-orthonormal frame of $TM$ over $U \subset M$, that is, a section of $\Prin_{\SO}(M,\gamma)|_U \to U$, and choose a lift to a section of $\Prin_{\Spin}(M,\gamma)|_U \to U$. 
This induces a smooth trivialisation $S|_U \cong U \times W$. 
In particular, each $\sigma \in W$  induces a smooth section of $S|_U = \Prin_{\Spin}(M,\gamma)|_U \times_{\Spin(n)} W$ which  is denoted again by $\sigma$ and called a {\em constant local spinor field} with respect to the chosen trivialisation of $\Prin_{\Spin}(M,\gamma)|_U$.

Let $\omega^g_{ji}  \in \Lloc^2  (U, T^*U)$, $1 \leq i,j \leq n$,  be the connection $1$-forms of the Levi-Civita connection $\nabla_g$ on $TM \to M$ with respect to the $g$-orthonormal frame $(e^g_1, \ldots, e^g_n)$ of $TM|_U$, that is, 
\[
     \nabla_g ( e^g_i)  =  \sum_{j=1}^n \omega^g_{ji} (e^g_j)  . 
\]
We now define 
\begin{equation} \label{def:spinconn} 
    \nabla_g^S  \sigma  :=   \frac{1}{2} \sum_{i < j} \omega^g_{ji} \; e_i  \cdot  e_j  \cdot  \sigma \in \Lloc^2(U, S|_U). 
\end{equation} 
Imposing the Leibniz rule, this extends to a metric connection  $\nabla^S_{g}$ on $S|_U \to U$ in the sense of Definition \ref{connection}.  
This definition is independent of the choice of $(e_1, \ldots, e_n)$ and its lift to $\Prin_{\Spin}(M,\gamma)|_U$ and can hence be used to define a global metric connection $\nabla^S_g$ on $S \to M$. 

\begin{dfn} We call $\nabla^S_g$ the {\em spinor connection} on $S \to M$ with respect to $g$. 
\end{dfn} 

\begin{rem} \label{translate} If $g$ is smooth, then $b_g$ induces a smooth map $ \beta_g \colon \Prin_{\SO} (M,\gamma) \to \Prin_{\SO}(M,g)$ of $\SO(n)$-principal bundles (compare  the proof of Lemma \ref{regb}) and there exists a spin structure $\Prin_{\Spin} (M,g) \to \Prin_{\SO}(M,g)$ together with a commutative diagram of smooth fiber bundles over $M$ 
\[
   \xymatrix{
    \Prin_{\Spin} (M,\gamma)   \ar[r]^{\tilde \beta_g}  \ar[d] & \Prin_{\Spin}(M,g)  \ar[d]   \\
    \Prin_{\SO} (M,\gamma)   \ar[r]^{\beta_g}                      & \Prin_{\SO}(M,g)   } 
\]
where $\tilde \beta_g$ is $\Spin(n)$-equivariant.

Let $S_g :=  \Prin_{\Spin}(M, g) \times_{\Spin(n)} W \to M$ be the associated smooth spinor bundle which is usually considered on the spin Riemannian manifold $(M,g)$. 
We obtain a smooth unitary vector bundle isomorphism 
\[
    \Xi_g \colon     S  =  \Prin_{\Spin} (M,\gamma) \times_{\Spin(n)} W \stackrel{ \tilde \beta_g \times_{\Spin(n)} \id}{ \longrightarrow } \Prin_{\Spin} (M,g) \times_{\Spin(n)} W = S_g 
\]
which is compatible with the $\cl(TM,\gamma)$-, respectively $\cl(TM,g)$-module structures on $S$ and $S_g$ intertwined by $\beta_g$. 

Hence the local expression of the spinor connection on $S_g$ from \cite[Theorem II.4.14]{LM89} pulls back to our defining equation \eqref{def:spinconn} under $\Xi_g$. 
\end{rem}

Let $E \to M$ be a Hermitian Lipschitz bundle  together with a Lipschitz connection $\nabla^E$  which we do not assume to be metric. 
The tensor product bundle $S \tensor E \to M$ is a Lipschitz bundle with induced inner product and connection
\begin{equation} \label{prod_conn} 
   \nabla^{S \otimes E}  = \nabla^S_{g}  \otimes 1 + 1 \tensor \nabla^E. 
\end{equation} 

\begin{dfn} \label{def:spin} 
We define the \emph{spin Dirac operator twisted with $(E,\nabla^E)$},  
\[
  \Dirac_E =  \Dirac_{g,E} \colon \Liploc(M,S \tensor E)\to  \Lloc^2 (M,S \tensor E), 
\]
locally over an open subset $U \subset M$ equipped with a smooth $\gamma$-orthonormal frame $(e_1, \ldots, e_n)$ of $TM$ and using Notation \ref{kurz}, by 
\begin{equation*} \label{def:Dirac} 
       \Dirac_{g,E } ( \psi) := \sum_{i=1}^n e_i    \cdot \nabla^{S \otimes E}_{e^g_i } \psi .  
\end{equation*} 
This expression is independent of the choice of $(e_1, \ldots, e_n)$ and hence gives a global definition of $\Dirac_{g,E}$. 
\end{dfn} 

From now on let $M$ be closed. 
Recall \cite[Lemma 3.1]{LL15} that the $L^2 $ and $W^{1,2}$-norms on $\Lip(M, S \otimes E)$, 
\[
\| \psi\|_{L^2}  = \int_M | \psi |^2 d \mu_g , \quad    \| \psi\|_{W^{1,2}}  = \int_M \bigl( | \psi|^2 + | \nabla^{S \otimes E} \psi|^2 \bigr) d\mu_g 
\]
are equivalent to the corresponding norms defined for the metric $\gamma$ (which induces a different spinor connection on $S$) and hence define the same $L^2$ and $W^{1,2}$-completions. 
If we wish to emphasize the precise norm, we add the relevant metric as a subscript.

We consider $\Dirac_{E}$ as an unbounded operator 
\[
  \Dirac_E \colon  L_\gamma^2(M , S \otimes E) \stackrel{\rm{dense}}{\supset}  \Lip(M, S\otimes E) \to L_\gamma^2(M, S \otimes E) . 
\]
Since integration by parts holds for Lipschitz sections, the twisted Dirac operator  $\Dirac_{E}$ has a formal adjoint  and is hence closable with closure 
\[
   \bar \Dirac_{E} \colon \dom ( \bar \Dirac_{E} ) \to L^2(M, S \otimes E). 
\]

\begin{prop} \label{domain} We have 
\[
    \dom ( \bar \Dirac_{E} ) = H^1( M , S \otimes E) . 
 \]
 \end{prop} 
 
\begin{proof} 
Fix local coordinates $(x^1, \ldots, x^n)$ on $U \subset M$, a smooth trivialisation $S|_U \cong U \times  \CC^q$ and a Lipschitz trivialisation $E|_U \cong U \times \CC^r$. 
We obtain an induced local trivialisation $(S \otimes E)|_{U} \cong U \times \CC^{q+r}$ with respect to which we can write 
\begin{equation} \label{local_dirac} 
    \Dirac_E  (u) = a^j  \partial_j u + b u 
\end{equation} 
where $a^j  \in \Wloc^{1,p}(U, \End(\CC^{q+r}) )$ and $b  \in \Lloc^p(U, \End(\CC^{q+r}) )$.
Hence  the operator $\Dirac_E$  satisfies the regularity conditions \cite[Equation (3.4)]{BartnikChrusciel}. 

Let $\sharp^{g_t} \colon T^*M \cong TM$ be the musical isomorphism for  $g$. 
By a standard computation, the principal symbol  $\sigma_{\xi}(\Dirac_E)$ for $\xi \in T^*M$ is given by 
\[
      \psi  \mapsto {\rm i}  \, \bigl(  \sharp^{g_t}( \xi ) \cdot \psi \bigr), 
\]
and hence the ellipticity condition \cite[Equation (3.5)]{BartnikChrusciel} is satisfied with a constant $\eta$ which can be chosen uniformly on $M$.

The assertion of Proposition \ref{domain} is hence implied by  \cite[Theorem 3.7]{BartnikChrusciel}. 
\end{proof}

\begin{prop}\label{L:formally self-adjoint}
If the connection $\nabla^E$ is metric, then the twisted Dirac operator $\Dirac_{E}$ is formally self-adjoint with respect to the $L^2$-inner product on $\Lip(M, S \otimes E)$ induced by $g$, that is, 
\begin{equation*}
    \left( \Dirac_{E} \psi_1 , \psi_2)_g =( \psi_1 ,  \Dirac_{E} \psi_2\right)_g , \quad \psi_1, \psi_2\in \Lip(M,S \tensor E). 
\end{equation*}
In particular the operator 
\[
    \bar \Dirac_{E} \colon H^1( M , S \otimes E) \to L_g^2(M , S \otimes E) 
\]
is self-adjoint. 
\end{prop}

\begin{proof}
For smooth $g$ this holds by the same argument as in~\cite[Proof of Lemma~5.1]{LM89}, since integration by parts holds for Lipschitz sections. 
If $g$ is of regularity $W^{1,p}$, $p >n$,  we choose a sequence of smooth metrics $g_\nu$ on $M$ converging to $g$ in the $W^{1,p}$-topology, hence in the $L^{\infty}$-topology. By Lemma \ref{regb} we obtain
\[
    \lim_{\nu \to \infty}   \left( \Dirac_{g_\nu ,E} \psi_1 , \psi_2\right)_{g_\nu}  = \left( \Dirac_{g,E} \psi_1, \psi_2\right)_g , \qquad
    \quad \lim_{\nu \to \infty}   \left(  \psi_1 , \Dirac_{g_\nu, E} \psi_2\right)_{g_\nu}  = \left( \psi_1, \Dirac_{g_\nu,E} \psi_2\right)_g . \qedhere
\]
\end{proof}

In the remainder of this section let 
\begin{enumerate}[\myicon] 
  \item $M$ be of even dimension $n$,
  \item $W$ be the unique irreducible complex $\cl_n$-module with the usual $\ZZ/2$-grading $W=W^+\oplus W^-$ and a $\Spin(n)$-invariant Hermitian metric, 
  \item $\Dirac_E = \bar\Dirac_{g,E} \colon H^1(M, S\otimes E) \to L^2(M, S \otimes E)$ be the closure of the twisted Dirac operator from Definition \ref{def:spin}. 
\end{enumerate} 
We obtain induced  $\ZZ/2$-gradings 
\[
    S^{\pm} = \Prin_{\Spin}(M, \gamma)  \times_{\Spin(n)}  W^{\pm}, \qquad S\tensor E=(S^+\tensor E)\oplus (S^-\tensor E), 
\]
and hence induced $\ZZ/ 2$-gradings 
\begin{align*} 
    L^2(M,S\tensor E) & =L^2(M,S^+\tensor E)\oplus L^2(M,S^-\tensor E), \\
    H^1(M,S\tensor E) & =H^1(M,S^+\tensor E)\oplus H^1(M,S^-\tensor E) 
\end{align*} 
such that the operator $ \Dirac_{E}$ is odd with respect to these gradings, that is,
 \begin{equation*}\label{E:barP odd}
  \Dirac_{E}  =\begin{pmatrix}0&  \Dirac_{E}^- \\  \Dirac_{E}^+&0 \end{pmatrix} . 
\end{equation*}

If the connection $\nabla^E$ on $E$ is metric, then the operators $ \Dirac_E^\pm\colon H^2(M,S^\pm\tensor E)\to L_g^2(M,S^\mp\tensor E)$ are adjoint to each other  by Proposition~\ref{L:formally self-adjoint}.

Let $\hat{\mathsf{A}}(M)$ denote the total $\hat{\mathsf{A}}$-class of the tangent bundle of $M$. 

\begin{thm} \label{thm:index} The operator $  \Dirac^+_{E} \colon H^1(M, S^+ \otimes E) \to L^2(M, S^- \otimes E)$ is Fredholm with index 
\[
      \ind( \Dirac^+_{E})  = \langle \hat{\mathsf{A}}(M) \cup \ch(E) , [M] \rangle. 
\]

\end{thm} 

Theorem \ref{thm:index} is well known for smooth $g$ and smooth twist bundles $E \to M$ by the Atiyah-Singer index theorem, see~\cite[Theorem III.13.10]{LM89}. 
We will reduce the general case to this case using the invariance of the index in norm continuous families of Fredholm operators.

In a first step, we put a smooth structure on the twist bundle $E \to M$. 

\begin{lem} \label{smooth} There is a smooth structure on $E \to M$ which is compatible with the given Lipschitz structure on $E$. 
\end{lem} 

\begin{proof} Let $\mu \colon M\to \mathbf{Gr}_r\left(\CC^N \right)$ be a Lipschitz map which classifies the bundle $E$ where $r=\rank E$ and $N$ is a large enough integer.
Then $\mu$ is homotopic, through a Lipschitz homotopy, to a smooth map $\tilde \mu \colon M\to \mathbf{Gr}_r\left(\CC^N\right)$.
The map $\tilde \mu$ classifies a smooth bundle $\tilde E \to M$ which is Lipschitz isomorphic to $E \to M$, and we can use this Lipschitz isomorphism to pull back the smooth structure on $\tilde E$ to a smooth structure on $E$. 
\end{proof}

Endow the smooth bundle $E \to M$ with a smooth inner product and a smooth connection $\tilde \nabla^E$ which remain fixed from now on.
We obtain families which depend continuously on $t \in [0,1]$ of
\begin{enumerate}[\myicon]  
   \item $W^{1,p}$-metrics $g_t := (1-t) g + t \gamma$ on $M$, 
   \item (not necessarily metric) Lipschitz connections $\nabla_t^E = (1-t) \cdot  \nabla^E  + t \cdot \tilde \nabla^E$ on $E$, 
   \item twisted Dirac operators $\Dirac_t = \bar \Dirac_{g_t,(E, \nabla^E_t)}  \colon H^1(M, S \otimes E) \to L^2 (M, S \otimes E)$. 
 \end{enumerate} 

Furthermore,  $a^j(t)  \in \Wloc^{1,p}(U , \End(\CC^{r+q}) )$ and $b(t) \in \Lloc^p(U, \End(\CC^{r+q}) )$ in Equation \eqref{local_dirac} depend continuously on $t$.

\begin{prop} \label{fredcont} 
  The following holds. 
  \begin{enumerate}[(a)]  
   \item $ \Dirac^+_t \colon H^1(M,S^+\tensor E)\to L^2(M,S^-\tensor E)$ is  bounded for $t \in [0,1]$ and the map $[0,1]\to \mathscr{B} (H^1(M,S^+\tensor E), L^2(M,S^-\tensor E))$, $t \mapsto  \Dirac^+_t$, is continuous.
  \item  $ \Dirac^+_t $ is Fredholm for $t \in [0,1]$. 
 \end{enumerate}
\end{prop}

\begin{proof}
The first statement follows from \cite[Theorem 3.7]{BartnikChrusciel} and the local uniform boundedness and continuity of the coefficients $a^j(t)$ and $b(t)$ in $t$. 
The Fredholm property follows from  \cite[Corollary 4.5]{BartnikChrusciel}.
\end{proof}

Theorem \ref{thm:index} now follows from Proposition \ref{fredcont} and the invariance of the Fredholm index of norm continuous families of Fredholm operators which implies
\begin{align*} 
   \ind( \Dirac^+_{g,(E, \nabla^E)}  ) & = \ind( \Dirac^+_0) = \ind ( \Dirac^+_1) \\
                                                         & = \ind (  \Dirac_{\gamma, (E, \tilde \nabla^E) })  = \langle \hat{\mathsf{A}}(M) \cup \ch(E) , [M] \rangle . 
\end{align*}

\section{An integral Schr\"odinger-Lichnerowicz formula}\label{sec:Lichnerowicz}

Let $(M, \gamma)$ be a closed smooth $n$-dimensional Riemannian spin manifold and let  $S \to M$ be a smooth spinor bundle as in Section \ref{S:Lipschitz bundles} associated to $\gamma$ and  some $\cl_n$-representation $W$ with $\Spin(n)$-invariant Hermitian inner product. 

Let $E \to M$ be a smooth Hermitian vector bundle with metric connection $\nabla^E$. 
If $g$ is a smooth metric on $M$, then the twisted Dirac operator $\Dirac_E = \Dirac_{g,E}$ from Definition \ref{def:spin}
satisfies the Schr\"odinger-Lichnerowicz formula \cite[Theorem~II.8.17]{LM89},
\[
    \Dirac_E^2 = \nabla_g^* \circ \nabla_g + \frac{1}{4} \scal_g + \mathcal{R}^E .
\]
Here  $\mathcal{R}^E$ is a self-adjoint bundle endomorphism of $S \otimes E$ depending on the curvature $R^E$ of $E$.

Now let $g$  be an admissible Riemannian metric on $M$. 
In this case an integral form of the  Schr\"odinger-Lichnerowicz formula still holds for the untwisted Dirac operator, see  \cite[Proposition 3.2]{LL15}. 
Here we will generalize this formula to Dirac operators twisted with a Lipschitz bundle $E \to M$ and draw some conclusions. 
Since in general the curvature of Lipschitz bundles over $M$ is not defined, we will work in the following setting.

Let $(N,h)$ be a smooth Riemannian manifold, set $\ell:=\dim N$ and let $\bigl(E_0,\nabla^{E_0}\bigr)$ be a smooth Hermitian vector bundle over $N$ with smooth metric connection $\nabla^{E_0}$. 
Take a Lipschitz map 
\[
     f \colon (M, d_g)  \to (N,d_h) 
\]
and let $E := f^*(E_0) \to M$ be the pull back bundle of $E_0$ under $f$. 
This is a Hermitian Lipschitz bundle with pull-back metric connection $\nabla^E= f^* \nabla^{E_0}$. 

Denote by $R^{E_0}\in\Omega^2(N,\End(E_0))$ the curvature of the connection $\nabla^{E_0}$ and let 
\[
    R^E = f^*(R^{E_0}) \in L^\infty \Omega^2(M,\End(E))
\]
be the pullback of $R^{E_0}$ along $f$.

For  $x \in M$ and $\sigma\tensor\eta\in (S\tensor E)_x$, we  set (recall Notation \ref{kurz}) 
\begin{equation}\label{eq:R^E_endomorphism}
    \mathcal R_g^E(\sigma\tensor\eta)\defeq \frac{1}{2}\sum_{i,  j=1}^n\bigl(e_i \cdot  e_j\cdot \sigma\bigr)\tensor\bigl(R^E_{e^g_i ,e^g_j }\eta\bigr)
\end{equation}
where $(e_1,\ldots, e_n)$ is some orthonormal basis  of $(T_x M, \gamma_x)$. 
Observe that $\mathcal R_g^E$ defines a section in $ L^\infty(M,\End(S\tensor E))$
and hence a bounded operator 
\[
   \mathcal R_g^E\colon L^2(M,S\tensor E)\to L^2(M,S\tensor E). 
\]

Let $\Dirac=  \bar \Dirac_{g, E} \colon H^1(M, S \otimes E) \to L^2(M, S \otimes E)$ be the self-adjoint spin Dirac operator on $M$ twisted with $(E, \nabla^E)$, see Proposition \ref{L:formally self-adjoint}. 
We denote by $\nabla$ the tensor product connection $\nabla^{S \otimes E}$ defined in~\eqref{prod_conn}.

\begin{thm}[Integral Schr\"odinger-Lichnerowicz formula] \label{T:integral Lichnerowicz}
For all $\psi_1$, $\psi_2\in H^1(M,S\tensor E)$ we get 
\begin{equation*}\label{eq:integral_Lichnerowicz}
    \bigl(\Dirac \psi_1, \Dirac \psi_2\bigr)_g=\bigl(\nabla \psi_1,\nabla \psi_2\bigr)_g+ \frac{1}{4} \bigl\llangle\scal_g,\langle \psi_1, \psi_2\rangle\bigr\rrangle+\bigl(\mathcal R_g^E\psi_1, \psi_2\bigr)_g . 
\end{equation*}
\end{thm}

The proof of Theorem  \ref{T:integral Lichnerowicz} is based on the following approximation result. 
\begin{lem}\label{L:C^0-H^1 approx}
Let $f\colon \R^n\to \R$ be a  Lipschitz function with compact support $K \subset \R^n$ and let $K \subset U \subset \R^n$ be an open neighborhood. 
Then there exists a sequence of smooth functions $f_\nu \colon \R^n\to\R$  with compact supports in $U$ and satisfying 
\begin{equation} \label{E:infty-L^2 approx} 
 \lim_{\nu \to \infty}  \bigl(  \norm{f_\nu -f}_{L^\infty} +\norm{f_\nu -f}_{H^1}\bigr)  = 0  , \qquad \max_{\nu} \norm{\dd f_\nu}_{L^\infty}  < \infty . 
\end{equation}
\end{lem}

\begin{proof}
  Let $\rho\colon\R^n\to[0,\infty)$ be a compactly supported smooth function such that $\int_{\R^n}\rho=1$. 
  For $\eps > 0$ set $\rho_\epsilon(x)\coloneqq\epsilon^{-n}  \rho(x/\epsilon)$ and consider the convolution
  \[
     \rho_\eps * f (x) = \int_{\R^n} \rho_{\eps} (\tau) \cdot f(x - \tau) d\tau . 
  \]
  For small enough $\eps$ we have $\supp( \rho_{\eps} * f) \subset U$. 
  Setting $f_\nu  :=  \rho_{1/ \nu} * f$ we have 
  \[
        \lim_{\nu \to \infty}  \| f_\nu -f\|_{L^\infty} = 0  , \qquad    \max_{k} \norm{\dd f_\nu }_{L^\infty} < \infty . 
   \]
 For a compactly supported function $u \colon \R^n \to \R$ we denote by $\hat u \colon \R^n \to \R$, 
 \[
     \hat u(\xi) = \int_{\R^n} e^{-{\bf i} \langle x, \xi \rangle} u(x) dx , 
 \]
 its Fourier transform. 
 We have $\hat\rho(0)=\int_{\R^n} \rho=1$ and $\widehat{\rho_\epsilon}(\xi)=\hat\rho(\epsilon\xi)$. 
 As $\hat\rho$ is a Schwartz function, there is a uniform (in $\epsilon$ and $\xi$) bound $C$ on $|\hat\rho_\epsilon|$.
  
  Let $\eta > 0$. 
  As $\hat f(\xi)$ is a Schwartz function, $|\hat f(\xi)|^2(1+\abs{\xi}^2)$ is integrable.
  Then there exists  $R > 0$ such that 
\[
   \int_{\R^n\setminus B_R}|\hat f(\xi)|^2(1+\abs{\xi}^2) d\xi < \frac{\eta}{2(C+1)^2} . 
\]
Set $F:= \max\{ \max_{\xi \in B_R} |\hat f (\xi)|^2 ( 1+ |\xi|^2), 1\}$. 
There exists $\eps_\eta > 0$ such that, for all $0 < \eps \leq \eps_\eta$, we get
\[
    \sup_{\xi \in B_r} \big( \hat\rho_\epsilon(\xi)-1\big)^2 < \frac{\eta}{2F\vol(B_R)} . 
 \]
  For $0 < \eps  \leq \eps_\eta$ we hence obtain, denoting by $\mu$ the Lebesgue measure on $\R^n$, 
    \begin{align*}
    \|\hat\rho_\epsilon & \hat f-\hat f\|^2_{L^2((1+\abs{\xi}^2)\mu)}\\
     &  \le   \int_{B_R} (\hat\rho_\epsilon-1)^2 |\hat f|^2(1+|\xi|^2)d\xi + \int_{\R^n\setminus B_R} ( | \hat\rho_\epsilon|+1)^2 |\hat f|^2 (1+\abs{\xi}^2)d\xi \\
     &  \le  \frac{\eta}{2 F \vol(B_R)}  \int_{B_R}  F  d\xi + (C+ 1)^2 \int_{\R^n \setminus B_R}  |\hat f|^2 (1+\abs{\xi}^2)d\xi \\
         &  \le    \frac{\eta}{2} + \frac{\eta}{2}   =  \eta .\\
\end{align*}
Using the definition of the $H^1$-norm in the Fourier picture, there is a constant $\Lambda > 0$ which only depends on $n$ such that  for $\eps > 0$ we have
 \[
     \|\rho_\epsilon*f- f\|^2_{H^1}   \leq  \Lambda \cdot \|\hat\rho_\epsilon  \hat f-\hat f\|^2_{L^2((1+\abs{\xi}^2)\mu)}  . 
 \]
Letting  $\eta$ go to $0$ in the previous estimates,  this shows $\lim_{\nu \to \infty} \| f_\nu  - f\|_{H^1} = 0$, as required. 
  \end{proof}

\begin{proof}[Proof of Theorem \ref{T:integral Lichnerowicz}] 
By Proposition \ref{cor:scal_on_sections}, it is enough to show the assertion for $\psi_1, \psi_2 \in \Lip(M, S \otimes E)$, which we assume from now on. 

\medskip 

\noindent \underline{Step 1:} 
Let $(g_\nu)$ a sequence of smooth metrics on $M$ with 
\[
      \lim_{\nu \to \infty} \| g_\nu - g\|_{W^{1,p}} = 0 . 
\]
We will show that if Theorem \ref{T:integral Lichnerowicz} holds for all $g_\nu$, then it also holds for $g$. 

Note that 
\begin{align*}
    \bigl(\Dirac_{g_\nu}\psi_1, \Dirac_{g_\nu }\psi_2\bigr)_{g_\nu} & \to\bigl(\Dirac_g \psi_1, \Dirac_g \psi_2\bigr)_g , \\
    \bigl(\nabla_{g_\nu }\psi_1,\nabla_{g_\nu } \psi_2\bigr)_{g_\nu}  & \to\bigl(\nabla_g \psi_1,\nabla_g \psi_2\bigr)_g , \\
    \bigl\llangle\scal_{g_\nu },\langle \psi_1, \psi_2\rangle \bigr\rrangle & \to\bigl\llangle\scal_g,\langle \psi_1, \psi_2\rangle \bigr\rrangle . 
\end{align*}   
Hence it remains to show that
 \[
        \bigl(\mathcal R_{g_\nu}^E \psi_1, \psi_2\bigr)_{g_\nu }\to\bigl(\mathcal R^E_g \psi_1, \psi_2\bigr)_g . 
\]
Let  $(e_1, \ldots, e_n)$ be a smooth $\gamma$-orthonormal frame of $TM$ over some open subset $U \subset M$.  
Without loss of generality, we may assume that $\psi_1$ is supported in a compact subset $K \subset U$. 
By the defining equation \eqref{eq:R^E_endomorphism} there is a constant $C$, not depending on $\nu$,  such that on $U$ we have 
\[
    \bigl|\mathcal R_{g_\nu}^E \psi_1-\mathcal R_g^E \psi_1\bigr| \leq     C \max_{i,j}\left|R^E_{e^{g_\nu}_i ,e^{g_\nu}_j}-R^E_{e^g_i,e^g_j}\right| |\psi_1|  \quad {\rm {a.e.}} 
\]

Since $p > n$, and hence we have $\lim_{\nu \to \infty} \| g_{\nu}  - g\|_{L^\infty(K)}=0$ and  $\lim_{\nu \to \infty} \|e^{g_\nu}_i - e^g_i\|_{L^\infty(K)} = 0$ for  $1 \leq i \leq n$ by Lemma \ref{regb},    we obtain   
\begin{equation*}
\quad \Bigl\|\frac{d\mu_{g_{\nu} }}{d\mu_\gamma}-\frac{d\mu_g}{d\mu_\gamma}\Bigr\|_{L^\infty(K)} \to 0,  \qquad \bigl\|R^E_{e^{g_\nu}_i,  e^{g_\nu}_j}-R^E_{e^g_i,e^g_j }\bigr\|_{L^\infty(K)} \to 0 , \quad 1 \leq i, j \leq n. 
\end{equation*}
We conclude
\begin{align*} 
    \big|   \bigl( \mathcal R^E_{g_{\nu} }  &  \psi_1,\psi_2\bigr)_{g_{\nu} } - \bigl( \mathcal R^E_g \psi_1, \psi_2\bigr)_g \big|   \\ 
                & = \big|  \int_M\left<\mathcal R^E_{g_{\nu} } \psi_1,\psi_2\right> \frac{d\mu_{g_{\nu} }}{d\mu_\gamma}d\mu_\gamma - \int_M\left<\mathcal R^E_g \psi_1,\psi_2\right> \frac{d\mu_g}{d\mu_\gamma}d\mu_\gamma \big| \\ 
  &  \leq \Big( C_n  \max_{i,j }\bigl\|R^E_{  e^{g_\nu}_i  , e^{g_\nu}_j } -R^E_{e^g_i,e^g_j}\bigr\|_{L^\infty(K)} \Bigl\|\frac{d\mu_{g_{\nu} }}{d\mu_\gamma}\Bigr\|_{L^\infty(K)} \\
       & \qquad \qquad +\bigl\|\mathcal R_g^E \bigr\|_{L^\infty(K)} \biggl\|\frac{d\mu_{g_{\nu} }}{d\mu_\gamma}-\frac{d\mu_g}{d\mu_\gamma}\biggr\|_{L^\infty(K) }  \Big) \int_M|\psi_1| \, |\psi_2| \,d\mu_\gamma.
\end{align*} 
For $\nu \to \infty$ this  tends to $0$, concluding Step 1. 

\medskip

\noindent \underline{Step 2:} We prove  Theorem \ref{eq:integral_Lichnerowicz} for smooth $g$ which together with Step 1 completes the proof of Theorem \ref{T:integral Lichnerowicz}. 
Choose open subsets $U \subset M$ and  $V \subset N$ with $f(U) \subset V$ and 
\begin{enumerate}[\myicon] 
  \item a smooth $\gamma$-orthonormal frame $(e_1, \ldots, e_n)$ of $TM|_U$, 
  \item local coordinates $(x^1, \ldots, x^n)$ on $U$ and $(y^1, \ldots, y^\ell)$ on $V$, 
  \item a  unitary trivialisation $E_0|_V \cong V  \times \CC^r$. 
\end{enumerate}

Since  both sides of~\eqref{eq:integral_Lichnerowicz} are sesquilinear forms, it suffices to show Theorem \ref{T:integral Lichnerowicz} for sections $\psi_1, \psi_2 \in \Lip(M, S \otimes E)$ which are supported in a compact subset $K \subset U$. 
With the given  trivialisation of $E_0|_V$ and a  continuity argument we can in fact assume that $\psi_1, \psi_2 \in C^{\infty}(U, S|_U \otimes \CC^r)$, supported in $K$.

Let $\omega_0 \in \Omega^1(V,M_r(\CC))$ be the smooth connection 1-form of $\nabla^{E_0}$ with respect to the local trivialisation $E_0 |_V \cong V\times \CC^r$ and define 
\[
   \omega := f^{\ast} \omega_0 \in L^\infty\Omega^1(U,M_r(\CC)). 
\]
Let 
\[
    \hat \Dirac\colon \Lip (U , S|_U \tensor\CC^r)\longrightarrow L^2 (U, S|_U \tensor \CC^r)
\]
be the spin Dirac operator on $U$ twisted with the trivial bundle $E|_U \cong U \times \CC^r$  endowed with the metric Lipschitz connection $\dd+\omega$.

Let $\Omega_0=\dd\omega_0+\omega_0\wedge\omega_0\in\Omega^2(V,M_r(\CC))$ be the smooth curvature form of the connection $\dd+\omega_0$ and set
\[
   \Omega := f^{\ast} \Omega_0 \in L^\infty\Omega^2(U,M_r(\CC)). 
\]
Let $\mathcal R\colon L^\infty(U, S|_U \tensor\CC^r)\to L^\infty(U, S|_U \tensor \CC^r)$ be defined by  
\begin{equation*}\label{E:R endomorphism}
    \mathcal R(\sigma\tensor\eta)\defeq \frac{1}{2}\sum_{i ,j =1}^n\bigl(e_i \cdot e_j \cdot\sigma\bigr)\tensor\bigl(\Omega_{e^g_i  , e^g_j } \, \eta\bigr)
\end{equation*}
on simple vectors $\sigma\tensor\eta \in S|_U \otimes \CC^r$.
Let $\hat \nabla$ be the tensor product connection $\nabla^S|_U \tensor 1+1\tensor (\dd+\omega)$ on $S|_U \times \CC^r$. 
It remains to show that 
\begin{equation}\label{E:local Lichnerowicz}
    \left(\hat \Dirac \psi_1,\hat \Dirac \psi_2\right)_g=\left(\hat \nabla \psi_1,\hat \nabla \psi_2\right)_g+\bigl\llangle\scal_g,\langle \psi_1, \psi_2\rangle\bigr\rrangle+\left(\mathcal R \, \psi_1, \psi_2\right)_g. 
\end{equation}

By Lemma \ref{L:C^0-H^1 approx} there is sequence of smooth functions $f_\nu  \colon U \to V$ such that, with respect to the chosen local coordinates on $U$ and $V$, 
\begin{equation} \label{annaehern} 
 \lim_{\nu \to \infty}  \bigl(  \norm{f_\nu -f}_{L^\infty(K) } +\norm{f_\nu -f}_{H^1(K) }\bigr)  = 0  , \qquad \max_{\nu} \norm{\dd f_\nu}_{L^{\infty}(K)}  < \infty . 
\end{equation}

Set  $\omega_\nu := f^\ast_\nu\omega_0 \in\Omega^1(U,M_r(\CC))$ and let 
\begin{equation*}
    \hat \Dirac_\nu\colon C_c^{\infty}(U, S|_U \tensor\CC^r)\longrightarrow C_c^\infty(U, S|_U \tensor\CC^r)
\end{equation*}
be the formally self-adjoint spin Dirac operator on $U$ twisted with the trivial bundle $E|_U \cong U \times \CC^r$ endowed with the smooth metric connection $\dd+\omega_\nu$.

Set $\Omega_\nu := f^{\ast} \Omega_0 \in \Omega^2(U, M_n(\CC))$ and let $\mathcal R_\nu \colon C^\infty(U, S|_U \tensor  \CC^r)\to C^\infty(U,\ S|_U \tensor \CC^r )$ be defined on simple tensors $\sigma \otimes \eta \in S|_U \otimes \CC^r$ by 
\begin{equation*}\label{E:R_k endomorphism}
    \mathcal R_\nu (\sigma\tensor\eta)\defeq \frac{1}{2}\sum_{i,j=1}^n\bigl(e_i\cdot e_j\cdot\sigma\bigr)\tensor \big( (\Omega_\nu)_{ e^g_i , e^g_j }\eta \big). 
\end{equation*}
Denote by  $\hat\nabla_\nu$ the smooth tensor product connection $\nabla^S|_U \tensor 1+1\tensor (\dd+\omega_\nu)$.

Using the classical Schr\"odinger-Lichnerowicz formula~\cite[Theorem~II.8.17]{LM89} and Remark \ref{translate}, we get for all $\nu$ that
\begin{equation}\label{E:approx integral Lichnerowicz}
    \left(\hat \Dirac_\nu \psi_1,\hat \Dirac_\nu \psi_2\right)_g =\left(\hat \nabla_\nu \psi_1,\hat \nabla_\nu \psi_2\right)_g+\bigl\llangle\scal_g,\langle \psi_1, \psi_2\rangle \bigr\rrangle+\left(\mathcal R_\nu\psi_1, \psi_2\right)_g  .
\end{equation}

In order to establish Equation \eqref{E:local Lichnerowicz}, it hence remains to show that
\begin{align}
 \label{toprove1}  \lim_{\nu \to \infty} \norm{\hat \nabla  \psi_\kappa-\hat \nabla_\nu \psi_\kappa}_{L^2(K)}   & = 0 , \quad \kappa = 1, 2, \\ 
 \label{toprove2} \lim_{\nu \to \infty}  \norm{ \mathcal  R \, \psi_1 - \mathcal R_\nu \, \psi_1}_{L^2(K)}   & =  0 . 
\end{align} 

Since $\hat \nabla -\hat\nabla_\nu=1\tensor\bigl( \omega- \omega_\nu \bigr)$ we have
\begin{align*}
    \norm{\hat \nabla \psi_\kappa-\hat \nabla_\nu \psi_\kappa}_{L^2(K)} \leq \norm{\psi_\kappa}_{L^\infty(K)}  \norm{\omega -\omega_\nu}_{L^2(K)} .
\end{align*}
Write $ \omega_0=\sum_{\alpha=1}^\ell \dd y^\alpha \tensor \Gamma_\alpha$  with $\Gamma_\alpha \in C^{\infty} ( V,  M_r(\CC))$.
With $f=(f^1,\ldots, f^\ell)$ and $f_\nu=(f^1_\nu,\ldots, f^\ell_\nu)$ we obtain 
\[
    \omega= \sum_{\alpha=1}^{\ell} \sum_{i=1}^n\left(\partial_i f^\alpha \right)\dd x^i \tensor \left(\Gamma_\alpha \circ f\right), \qquad \omega_\nu = \sum_{\alpha = 1}^\ell \sum_{i=1}^n \bigl(\partial_if^\alpha_\nu \bigr)\dd x^i\tensor \left(\Gamma_\alpha \circ f_\nu\right).
\]
Setting $\Gamma = (\Gamma_1, \ldots, \Gamma_\ell)$ we hence obtain, a.e.{} over $K$, that 
\[
    | \omega-\omega_\nu|    \leq  |\dd f-\dd f_\nu|  \cdot   \norm{\Gamma \circ f }_{L^\infty(K)}  +|\dd f_\nu| \cdot \norm{\Gamma \circ f  - \Gamma \circ f_\nu}_{L^\infty(K)} .
\]
By the triangle inequality for the $L^2$-norm this implies 
\[
    \norm{\omega-\omega_\nu}_{L^2(K)}   \leq  \norm{\Gamma \circ f}_{L^\infty(K)}  \norm{\dd f-\dd f_\nu}_{L^2(K)} +  \norm{ \Gamma \circ f - \Gamma \circ f_\nu}_{L^\infty(K)}  \norm{\dd f_\nu}_{L^2(K)} .
\]
With \eqref{annaehern}  the last expression  tends to $0$ as $\nu\to\infty$, proving \eqref{toprove1}. 

For \eqref{toprove2}, we write
\[
    \Omega_0=\sum_{1 \leq \alpha <\beta \leq \ell} ( \dd y^\alpha \wedge\dd y^\beta ) \tensor B_{\alpha,\beta}, \quad B_{\alpha,\beta} \in C^{\infty}(V, M_r(\CC)) , 
\]
and obtain
\begin{align*} 
    \Omega & =\sum_{\alpha<\beta}\sum_{i<j}\Bigl[ \bigl(\partial_i f^\alpha\bigr)\bigl(\partial_j f^\beta\bigr)-\bigl(\partial_j f^\alpha \bigr)\bigl(\partial_i f^\beta\bigr)\Bigr]  \dd x^i \wedge\dd x^j  \tensor \bigl(B_{\alpha,\beta }\circ f\bigr), \\
  \Omega_\nu & =\sum_{\alpha<\beta}\sum_{i<j}\left[\bigl(\partial_i f^\alpha_\nu \bigr)\bigl(\partial_j f^\beta_\nu \bigr)-\bigl(\partial_j f^\alpha_\nu \bigr)\bigl(\partial_i f^\beta_\nu \bigr)\right]  \dd x^i \wedge\dd x^j  \tensor \bigl(B_{\alpha,\beta}\circ f_\nu\bigr).
\end{align*} 
Therefore we get 
\[
    ( \mathcal R - \mathcal R_\nu) ( \sigma \tensor \eta)   \defeq \frac{1}{2}\sum_{i,j=1}^n\bigl(e_i \cdot e_j \cdot\sigma\bigr)\tensor\bigl(\Delta^{(\nu)}_{e^g_i, e^g_j }\eta  + \Lambda^{(\nu)}_{e^g_i, e^g_j }\eta\bigr)
\]
where
\begin{align*} 
    \Delta^{(\nu)} & \defeq \sum_{\alpha<\beta}\sum_{i<j}  \left[\bigl(\partial_i f^\alpha_\nu \bigr)\bigl(\partial_j f^\beta_\nu \bigr)-\bigl(\partial_j f^\alpha_\nu \bigr)\bigl(\partial_i f^\beta_\nu \bigr)\right] \dd x^i\wedge\dd x^j \tensor \Bigl[\bigl(B_{\alpha,\beta}\circ f\bigr)-\bigl(B_{\alpha,\beta }\circ f_\nu \bigr)\Bigr], \\
    \Lambda^{(\nu)} & \defeq \sum_{\alpha<\beta}\sum_{i<j }Z_{i,j}^{\alpha,\beta} \, \dd x^i\wedge\dd x^j  \tensor \bigl(B_{\alpha,\beta}\circ f\bigr), \\
    Z_{i ,j}^{\alpha,\beta} & \defeq \big( \mathop{\partial_i f^\alpha} \mathop{\partial_j f^\beta}-\mathop{\partial_j f^\alpha}\mathop{\partial_i f^\beta}\big) -\big( \mathop{\partial_i f^\alpha_\nu }\mathop{\partial_j f^\beta_\nu }-\mathop{\partial_j f^\alpha_\nu }\mathop{\partial_i f^\beta_\nu }\big) .
\end{align*} 
It remains to show that
\begin{equation}\label{E:2 limits}
 \lim_{\nu \to \infty}    \norm{ \Delta^{(\nu)}}_{L^2(K)}     =   0 , \qquad \lim_{\nu \to \infty} \norm{  \Lambda^{(\nu)} }_{L^2(K)}  =  0. 
\end{equation}

Since $\big|  \bigl(\partial_i f^\alpha_\nu \bigr)\bigl(\partial_j f^\beta_\nu \bigr)-\bigl(\partial_j f^\alpha_\nu \bigr)\bigl(\partial_i f^\beta_\nu \bigr)\big| \leq 2 | \dd f_\nu|^2$ a.e., there exists a constant $C$, not depending on $\nu$, such that, with $B:= ( B_{\alpha,  \beta})_{1 \leq \alpha,\beta \leq \ell}$, we obtain
\begin{equation*}
    \norm{\Delta^{(\nu)} }_{L^2(K)}  \leq C \norm{\dd f_\nu}^2_{L^\infty(K) } \norm{(B \circ f)-(B \circ f_\nu)}_{L^\infty(K)} .
\end{equation*}
From~\eqref{annaehern}  it follows that the last expression  tends to $0$ as $\nu\to\infty$, proving the first part of \eqref{E:2 limits}. 

For the second part in~\eqref{E:2 limits} we write
\[
  Z_{i,j }^{\alpha,\beta} = \mathop{\partial_i f^\alpha}  \left(\mathop{\partial_j f^\beta}-\mathop{\partial_j f^\beta_\nu }\right)-\mathop{\partial_j f^\alpha}\left(\mathop{\partial_i f^\beta}-\mathop{\partial_i f^\beta_\nu }\right) 
  +\left(\mathop{\partial_i f^\alpha}-\mathop{\partial_i f^\alpha_\nu }\right)\mathop{\partial_j f^\beta_\nu }-\left(\mathop{\partial_j f^\alpha}-\mathop{\partial_j f^\alpha_\nu }\right)\mathop{\partial_i f^\beta_\nu} 
\]
from which
\begin{equation*}
    \big| Z^{\alpha,\beta}_{i,j}\big| \leq 2\bigl( | \dd f | +  | \dd f_\nu | \bigr)  \cdot |\dd f -\dd f_\nu | \quad {\rm {a.e.}} 
\end{equation*}
We deduce that there exists a constant $C$, not depending on $\nu$, such that
\begin{equation*}
    \norm{ \Lambda^{(\nu)} }_{L^2(K)}  \leq C \bigl(\norm{\dd f}_{L^\infty(K)} +\norm{\dd f_\nu}_{L^\infty(K)}  \bigr)\norm{\dd f-\dd f_\nu}_{L^2(K)}  \norm{B\circ f}_{L^{\infty}(K) }    .
\end{equation*}
Again using~\eqref{annaehern}  it follows that the last expression  tends to $0$ as $\nu\to\infty$ so that the second equation in \eqref{E:2 limits} holds as well. 

Hence the proof of Theorem \ref{T:integral Lichnerowicz} is complete. 
\end{proof} 

In the remainder of this section we apply the Schr\"odinger-Lichnerowicz formula to harmonic spinor fields under lower bounds of $\scal_g$ and $\mathcal R^E$. 
From now on let $M$ be connected. 

\begin{dfn} \label{lowerbounds}
For a function $\vartheta \in L^\infty(M)$, we denote by 
\[
   \mathcal I_\vartheta \colon L^1(M) \to \R , \quad u \mapsto \int_M \vartheta  u \, d\mu_g,
\]
the associated regular distribution.
\begin{enumerate}[\myicon] 
\item We say that $\scal_g \geq \vartheta$ {\em in the distributional sense}  if for all $u\in C_c^{\infty}(M)$, $u \geq 0$,  we have
\[
    \llangle\scal_g , u \rrangle \geq \mathcal I_\vartheta(u) . 
\]
\item We say that $\mathcal R_g^E \geq \vartheta$ \emph{fiberwise in the distributional sense} if for all $\psi \in \Lip(M, S \otimes E)$ we get 
  \[
    \left(\mathcal R_g^E \psi , \psi\right)_g  \geq \mathcal I _\vartheta( |\psi|^2 )  . 
  \]

\end{enumerate} 
\end{dfn}

\begin{prop}\label{T:kerP}
Suppose there exists $\vartheta \in L^\infty(M)$ with 
\begin{align}
    \label{E:vanishing2}   \frac{1}{4} \scal_g & \geq \vartheta ,\\
   \label{E:vanishing1}  \mathcal R_g^E & \geq - \vartheta .
\end{align}
Then the following assertions hold. 
\begin{enumerate}[(a)] 
\item \label{partA} For each $\psi \in \Ker ( \Dirac_{g,E})$ the norm  $|\psi| \in H^1(M)$ is constant a.e.
\item \label{partB} If $\Ker ( \Dirac_{g,E}) \neq 0$, then   $\frac{1}{4} \scal_g =\mathcal I _\vartheta$. 
\end{enumerate} 
\end{prop}

\begin{proof}
Let $\psi \in \Ker\left(\Dirac_{g,E}\right) \subset H^1(M ,S \otimes E)$.
From Theorem \ref{T:integral Lichnerowicz} and our assumptions we obtain, using $L^2$-norms with respect to the metric $g$ on $M$, 
\[
    0 = \|\nabla_g \psi \|^2_{L^2}+ \frac{1}{4} \bigl \llangle \scal_g ,| \psi|^2  \bigr\rrangle+\bigl(\mathcal R^E \psi, \psi\bigr)_g \geq  \| \nabla_g \psi\|_{L^2}^2+ \mathcal I _\vartheta(|\psi|^2) + \mathcal I _{-\vartheta} ( |\psi|^2)   \geq \norm{\nabla_g \psi }_{L^2}^2.
\]
This implies  $\norm{\nabla_g \psi}_{L^2}=0$ and hence $\nabla_g \psi=0$ a.e.
Observe that  $|\psi|^2  \in W^{1,1}(M)$.
Since $\nabla_g$ is a metric connection, we furthermore have
\[
    \dd|\psi |^2=2\left<\nabla_g \psi, \psi\right>_{S \otimes E}=0 \quad \rm{a.e.}
\]
and hence $|\psi|^2 \in W^{1,1}(M)$ with $\dd |\psi|^2=0$ a.e. 
Since $M$ is connected, we deduce that $|\psi|^2$ is constant, finishing the proof of \eqref{partA}.

For \eqref{partB} it remains to show that for all smooth  $u\colon M\to[0,1]$ we get $\frac{1}{4} \left\llangle\scal_g,u\right\rrangle \leq \mathcal I _{ \vartheta}(u)$. 
For a contradiction, assume that there exists $\eps > 0$ and a smooth $u\colon M\to[0,1]$ with 
\begin{equation*}\label{eq:scal_pos}
\frac{1}{4} \left\llangle\scal_g,u\right\rrangle= \mathcal I _{ \vartheta}(u) + \eps   .
\end{equation*}
Let  $\psi \in \Ker\bigl(\Dirac_{g,E}\bigr)$ be a non-zero $H^1$-section. 
By  \eqref{partA} there exists a constant $C > 0$ with $|\psi|^2\equiv C =C u+C(1-u)$ a.e. 
Both $u$ and $1-u$ are non-negative and hence
\begin{align*} 
    0 &= (\Dirac_{g,E} \psi,\Dirac_{g,E}\psi) \\
     &\stackrel{\rm{Thm.} \eqref{eq:integral_Lichnerowicz}}{=}  \underbrace{\| \nabla_g\psi\|_{L^2}^2}_{\ge 0} + \underbrace{\left\llangle \mathcal R^E_g \psi, \psi \right\rrangle +\mathcal I _\vartheta(|\psi|^2)}_{\ge 0 \text{ by  \eqref{E:vanishing1}  }} +  \frac{1}{4}  \left\llangle \scal_g, |\psi|^2 \right\rrangle -\mathcal I _\vartheta(|\psi|^2) \\
    &\geq C\bigl(\underbrace{  \frac{1}{4} \left\llangle\scal_g ,u\right\rrangle -\mathcal I _\vartheta(u)}_{=\eps} + \underbrace{ \frac{1}{4}  \left\llangle\scal_g,1-u\right\rrangle  -\mathcal I _\vartheta (1-u)  }_{\ge 0\text{ by \eqref{E:vanishing2} }} \bigr) \geq C \eps>0,
  \end{align*} 
a contradiction.
\end{proof}

\section{Proof of Theorem \ref{T:main_even} and \ref{theo:disk}}\label{sec:proofs}

\begin{proof}[Proof of Theorem \ref{T:main_even}]
Pick a smooth Riemannian metric  $\gamma$ on $M$ and  let  $S \to M$ be the smooth spinor bundle associated to $\gamma$, some spin structure on $(M, \gamma) $  and the irreducible complex $\cl_n$-module $W$ as in Section \ref{S:Lipschitz bundles}.
Furthermore, let $\Sigma \to \S^n$ be the smooth spinor bundle associated to the standard round  metric $g_0$ on $\S^n$, the unique spin structure on $(\S^n, g_0) $ and the $\cl_n$-module $W$. 
Both spinor bundles $S$ and $\Sigma$  are equipped with the respective metric spinor connections. 
Since $n$ is even we have the $\ZZ/2$-grading $\Sigma=\Sigma^+\oplus\Sigma^-$. 
By pulling back $\bigl(\Sigma,\nabla^{\Sigma}\bigr)$ along $f$ we obtain the Hermitian Lipschitz bundle $ E \to M$ with induced  metric  connection $\nabla^E$.
Let 
\[
     \Dirac_{g,E} \colon H^1 \left(M, S \otimes E\right) \to L^2\left(M, S \otimes E\right)
\]
be the  self-adjoint spin Dirac operator on $M$ for the twist bundle  $\bigl(E,\nabla^E\bigr)$.
Since $E = E^+ \oplus E^-$ is also $\ZZ/2$-graded, we obtain an induced operator
\[
    D^{+} _{g,E} \colon H^1\left(M, (S^+ \otimes E^+) \oplus (S^- \otimes E^-)\right) \to L^2\left( M , (S^- \otimes E^+) \oplus (S^+ \otimes E^-) \right) . 
\]
According to Theorem~\ref{thm:index}, the index is equal to 
\begin{equation} \label{index} 
    \ind\bigl(\Dirac_{g,E}^{+}\bigr) = \left\langle \hat{\mathsf{A}} (M) \cup f^* \left(  \ch(\Sigma^+) - \ch(\Sigma^-)\right) , [M] \right\rangle . 
\end{equation} 
The Chern character difference is computed in  \cite[ Prop.~III.11.24]{LM89}. 
This calculation shows 
\[
    \ch(\Sigma^+) - \ch(\Sigma^-) = (-1)^{\tfrac{n}{2}} \, e(T \S^n), 
 \]
 where $e(T\S^n) \in H^{n}(\S^n ; \mathbb{Q})$ is the Euler class of the tangent bundle $T\S^n \to \S^n$.
 Plugging this result into Equation \eqref{index} and using that the $H^0(M ; \mathbb{Q})$-component of $\hat{\mathsf{A}}(M)$ is equal to $1$, we obtain 
 \begin{align*} 
 (-1)^{\tfrac{n}{2}} \cdot   \ind\bigl(\Dirac_{g,E}^{+}\bigr) & = \langle \hat{\mathsf{ A}}(M) \cup f^*\left(e(T\S^n)\right) , [ M ] \rangle  \\ & =  \langle f^* \left(e(T\S^n)\right) , [M] \rangle \\ & =  \langle e(T\S^n) , f_*([M]) \rangle \\ & =  \deg (f) \cdot \underbrace{\langle e(T\S^n) , [\S^n] \rangle}_{=2}  . 
 \end{align*} 
Since $\deg(f) \neq 0$ by assumption, the last expression is non-zero, and using that  $\Dirac_{g,E}$ is self-adjoint, we conclude that there exists $0 \neq \psi  \in \ker \Dirac_{g,E}$.

We now analyze the term $\mathcal R_g^E$ in the Schr\"odinger-Lichnerowicz formula in Theorem \ref{T:integral Lichnerowicz}. 
For the next proposition, recall the shorthand $v^g = b_g(v)$ from Notation \ref{kurz}. 
%For all points $x \in M$ where $f$ is differentiable and all $v \in T_x M$, we will use  the notation 
%\[
%      \bar v_x := d_x f (v^g)    \in  T_{f(x)} \S^n . 
%\]

\begin{prop} \label{wonderfulestimates}
Let $x \in M$ be a point where $f$ is differentiable.  
Assume that either $|d_x f| \leq 1$, or $n \geq 4$ and $|\Lambda^2 d_x f| \leq 1$.

Then for all $\omega \in (S \otimes E)_x$ we have 
\[
    \langle   \mathcal R_g^E \omega, \omega \rangle \geq - \frac{1}{4} n(n-1) |\omega|^2 . 
\]
Furthermore, for $\omega\neq 0$ equality holds if and only if $d_x f \colon (T_x M, g_x) \to T_{f(x)} \S^n$ is an isometry and 
\[
     \big( v \cdot w \otimes  d_x f(v^g)  \cdot d_x f(w^g)   \big) \cdot \omega = \omega 
\]
for all $\gamma$-orthonormal vectors $v, w \in T_x M$. 
\end{prop} 

\begin{proof} 
By a singular value decomposition of $d_x f \circ b_g$, there exists a $\gamma$-orthonormal basis $(e_1, \ldots e_n)$ of $T_x M$, a $g_0$-orthornomal basis $(\eps_1, \ldots, \eps_n)$ of $T_{f(x)} \S^n$ and real numbers $\mu_i \geq 0$, $1 \leq i \leq n$, with $d_x f (e^g_i)  = \mu_i \cdot \eps_i$.

Since the curvature operator of $(\S^n,g_0) $ acts as the identity on $2$-forms, formula \cite[Equation (4.37) in Chapter II]{LM89} (also compare \cite[Lemma 4.3]{Lla98})  implies 
\[
    R^{E}_g \omega= \frac{1}{2} \sum_{i,j = 1}^n \mu_i \mu_j \, \big( e_i \cdot e_j  \otimes R^{\Sigma}_{\eps_i , \eps_j} \big)  \,  \omega  = \frac{1}{4} \sum_{i \neq j}  \mu_i \mu_j \, \big( e_i \cdot e_j \otimes \eps_j  \cdot \eps_i  \big) \cdot \omega . 
\]
Each  Clifford multiplication operator $e_i \cdot e_j \otimes \eps_j \cdot \eps_i \colon (S \otimes E)_x \to (S  \otimes E)_x$ is a self-adjoint involution, hence of norm $1$, so that 
\[
    \langle(  e_i \cdot e_j \otimes \eps_j \cdot \eps_i) \cdot \omega, \omega \rangle \geq - | \omega|^2 . 
\]
We conclude 
\[
    \langle  R^{E}_g \omega , \omega \rangle =  \frac{1}{4} \sum_{i \neq j}  \mu_i \mu_j \, \langle (e_i \cdot e_j \otimes \eps_j \cdot \eps_i)\cdot \omega, \omega \rangle \geq - \frac{1}{4} \sum_{i \neq j} \mu_i \mu_j | \omega|^2 \geq - \frac{1}{4} n(n-1) |\omega|^2 .  
\]
Furthermore, for $\omega \neq 0$ equality holds if and only if 
\begin{enumerate}[(a)]
   \item \label{erstens} $\mu_i = 1$ for all $1 \leq i \leq n$, that is, $d_x f \circ b_g \colon (T_x M , \gamma_x) \to T_{f(x)} \S^n$ is an isometry and 
   \item \label{zweitens} $( e_i \cdot e_j \otimes \eps_i \cdot \eps_j ) \cdot \omega = \omega$ for $1 \leq i \neq j \leq n$. 
\end{enumerate} 
For  \eqref{erstens} we use that either $|d_x f| \leq 1$ (hence $\mu_i \leq 1$ for $1 \leq i \leq n$), or $n \geq 4$ and $|\Lambda^2 d_x f| \leq 1$ (hence $\mu_i \mu_j \leq 1$ for all $1 \leq i \neq j \leq n$).

Since $b_g \colon (T_x M, \gamma_x) \to (T_x M, g_x)$ is an isometry, assertions \eqref{erstens} and \eqref{zweitens} are equivalent to the conditions stated in Proposition \ref{wonderfulestimates}. 
 \end{proof}

Since $f$ is differentiable almost everywhere, Proposition \ref{wonderfulestimates} implies 
\begin{equation*} 
  \mathcal R^E  \geq-\frac{1}{4}  n (n-1)
\end{equation*} 
fiberwise  in the distributional sense. 
Together with our assumption $\scal_g\geq n (n-1)$ in the distributional sense, Proposition \ref{T:kerP} with $\vartheta := \frac{1}{4} n(n-1)$  implies that there exists $C > 0$ with $|\psi| =  C$ a.e.{} and $\scal_g =  \mathcal{I}_{n (n -1)}$.

From Theorem \ref{eq:integral_Lichnerowicz} we hence obtain that
\begin{equation*}
    0=\|\Dirac_{g,E}\psi\|_{L^2}^2  \geq \frac{1}{4}   \left\llangle\scal_g,|\psi|^2\right\rrangle + \left(\mathcal R^E \psi,\psi \right)_g  = \frac{ C^2}{4}  n(n-1) \vol(M,g) + \left(\mathcal R^E \psi,\psi \right)_g . 
 \end{equation*}

Using  the equality statement of  Proposition \ref{wonderfulestimates},  this implies that at almost all points $x$ where $f$ is differentiable, the map  $d_x f \colon (T_x M, g_x) \to T_{f(x)} \S^n$ is an isometry and 
\begin{equation} \label{Clifford} 
     \big( v \cdot w \otimes  d_x f(v^g)  \cdot d_x f(w^g)   \big) \cdot  \psi(x) = \psi(x)  \textrm{ for all } \gamma\textrm{-orthonormal }  v, w \in T_x M. 
\end{equation} 
Let $\Mr \subset M$ be the subset of full measure of all  $x \in M$ where $f$ is differentiable and 
\begin{enumerate}[\myicon] 
\item $ d_x f \colon (T_x M, g_x) \to T_{f(x)} \S^n$ is an isometry, 
\item  $\big( v \cdot w \otimes  d_x f(v^g)  \cdot d_x f(w^g)   \big) \cdot  \psi(x) = \psi(x)$ for all $\gamma$-orthonormal  $v, w \in T_x M$,
\item  $|\psi(x)| = C > 0$, where $C$ is independent of $x$.  
\end{enumerate}

\smallskip 

\begin{prop}\label{prop:orientation}
  The differential $d_x f$ is either orientation preserving at almost all
  $x \in \Mr$ or orientation reversing at almost all $x \in \Mr$.
\end{prop}
\smallskip 

Together with the previous discussion, this implies,  possibly after reversing the orientation of $M$, that the differential $df$ is an orientation preserving isometry a.e.{}  on $M$ and the proof of Theorem \ref{T:main_even}  is completed by Theorem  \ref{essrigid}. 

For proving Proposition \ref{prop:orientation}, let 
\[
    M_{\pm} := \{ x \in \Mr \mid  \det ( d_x f)  = \pm 1 \} \subset \Mr 
\]
be the subset of $\Mr$ where $df$ is orientation preserving or orientation reversing, respectively. 

Consider the Clifford algebra bundles $\Cl(M) \to M$ for $(M,\gamma)$ and $\Cl(\S^n) \to \S^n$ for $(\S^n, g_0)$.  
The  {\em oriented volume elements} 
\[
    \vol_\gamma = e_1 \cdot  \ldots \cdot e_n \, , \qquad  \vol_{g_0} = \eps_1 \cdot \ldots \cdot \eps_n , 
\]
where $(e_1, \ldots, e_n)$ is a local smooth positively oriented $\gamma$-orthonormal frame of $TM$ and  $(\eps_1, \ldots, \eps_n)$ is a local smooth positively oriented $g_0$-orthonormal frame of $T\S^n$, define global smooth sections of $\Cl(M)$ and $\Cl(\S^n)$. 
Hence we obtain a self-adjoint Lipschitz involution of the Lipschitz bundle $S \otimes E \to M$ which over $x \in M$ acts by left Clifford multiplication with $\vol_\gamma(x) \otimes \vol_{g_0} (f(x))$. 
Let $S\otimes E = W^+\oplus W^-$ be the resulting orthogonal splitting into $\pm1$-eigenbundles and let  
\[
    \pi_+ \in \Lip ( M  , \End(S\otimes E))
\]
be the orthogonal projection onto $W^+$. 

Now take $x \in \Mr$ and a positively oriented $\gamma$-orthonormal basis  $(e_1, \ldots, e_n)$ of $T_x M$. 
Then $(d_x f( e^g_1) , \ldots, d_x f(e^g_n))$ is a $g_0$-orthonormal basis of $T_{f(x)} \S^n$ which is positively oriented, if and only if $d_x f \circ b_g \colon T_x M \to T_{f(x)}\S^n$ is orientation preserving. 
The last condition is equivalent to $x \in M_+$ as  the $\gamma$-self-adjoint map $b_g \colon T_x M \to T_x M$ (defined in Lemma \ref{regb}) is positive.

Applying \eqref{Clifford} iteratively for $(v,w) = (e_1, e_2), (e_3, e_4), \ldots, (e_{n-1} , e_n)$, using that $n$ is even, we get
\[
    \big( ( e_1 \cdot \ldots \cdot e_n)   \otimes  (d_x f( e^g_1) \cdot  \ldots \cdot d_x f(e^g_n)  ) \big) \cdot \psi(x) = \psi(x)  , 
\]
and together with the previous discussion, this shows 
\[
  \big( \vol_\gamma(x)  \otimes \vol_{g_0} (f(x))  \big) \cdot \psi (x) =  \pm \psi(x) \textrm{ if } x\in M_{\pm} . 
\]
Consequently,
  \begin{equation*}
\pi_+\psi(x)=  \begin{cases} \psi(x) & \textrm{ for all } x \in M_+, \\
                                                   0 & \textrm{ for all } x\in M_-.
\end{cases}
\end{equation*}
As $\psi \in W^{1,2}(M, S\otimes E)$ and $\pi_+$ is Lipschitz, we have $\abs{\pi_+\psi}\in W^{1,2}(M, \R)$, and this function  is identically $0$ on $M_-$ and constant with non-zero value $C$ on $M_+$. 
Because $M$ is connected and the map is of
  Sobolev regularity   $W^{1,2}$, either $M_-$ or $M_+$ has measure zero, which finishes the proof of Proposition \ref{prop:orientation}. 

\end{proof}

\begin{proof}[Proof of Theorem \ref{theo:disk}]
  Choose a smooth collar neighborhood of $\partial M \subset M$ and let $(\hat M, \hat g)$ be the smooth double of $(M,g)$ with reflected metric. 
  The metric $\hat g$ on $\hat M$ is Lipschitz and hence admissible. 
 Since the boundary has positive mean curvature, \cite[Proposition 5.1]{LL15} implies that  $\scal_{\hat g} \geq n(n-1)$ in the distributional sense.   
 
 We define $\hat f \colon \hat M\to \S^n$ to be equal to $f$ on the first copy of $M$ in $\hat M$ and to be equal to $\rho \circ f$ on the second copy of $M$, where $\rho$ was defined in Example \ref{folding}. 
 Since $\rho$ is $1$-Lipschitz and $f(\partial M) \subset \D^n_-$ by assumption,  the map $\hat f \colon \hat M \to \S^n$ is Lipschitz. 
 Furthermore, if $n \geq 4$, then $d\hat f$ is area-nonincreasing a.e., and if $n =2$, then $\hat f$ is $1$-Lipschitz. 

The map   $\hat f$ sends the second copy of $M$ in $\hat M$ to  the lower hemisphere $\D_-$, hence we get 
\[
   \deg \hat f = \deg \left(f\colon  (M, \partial M)\to (\S^n,\D^n_-) \right) \neq 0 . 
\]
Theorem \ref{T:main_even} now implies that $\hat f$ is a metric isometry.
We conclude, using smoothness of $g$, that $\im (f) = \D^n_+$ and  that $f \colon (M, g)  \to \D^n_+$ is a smooth Riemannian isometry. 
\end{proof} 

\bibliographystyle{abbrv}
\bibliography{biblio}

\end{document}